\documentclass{article}
\usepackage[T1]{fontenc}
\usepackage[english]{babel}
\usepackage[pdfusetitle]{hyperref}
\usepackage{authblk}
\usepackage{amsmath,amssymb,amsthm,graphicx,stmaryrd,enumerate,cleveref,latexsym,makeidx,mathtools}
\makeindex

\newcommand{\Alpha}{\mathrm{A}}
\newcommand{\Kappa}{\mathrm{K}}
\newcommand{\assign}{\coloneqq}
\newcommand{\backassign}{\eqqcolon}
\newcommand{\llangle}{{\langle\!\langle}}
\newcommand{\mathd}{\mathrm{d}}
\newcommand{\nin}{\not\in}
\newcommand{\npreccurlyeq}{\mathrel{\not\preccurlyeq}}
\newcommand{\nsucccurlyeq}{\mathrel{\not\succcurlyeq}}
\newcommand{\of}{:}
\newcommand{\rrangle}{{\rangle\!\rangle}}
\newcommand{\suchthat}{:}
\newcommand{\tmabbr}[1]{#1}
\newcommand{\tmem}[1]{{\em #1\/}}
\newcommand{\tmmathbf}[1]{\ensuremath{\boldsymbol{#1}}}
\newcommand{\tmop}[1]{\ensuremath{\operatorname{#1}}}
\newcommand{\tmstrong}[1]{\textbf{#1}}
\newcommand{\tmtextit}[1]{\text{{\itshape{#1}}}}
\newcommand{\tmtextup}[1]{\text{{\upshape{#1}}}}
\newenvironment{descriptioncompact}{\begin{description} }{\end{description}}
\newenvironment{enumeratealpha}{\begin{enumerate}[a{\textup{)}}] }{\end{enumerate}}
\newenvironment{enumerateroman}{\begin{enumerate}[i.] }{\end{enumerate}}
\newenvironment{itemizedot}{\begin{itemize} }{\end{itemize}}
\newtheorem{theorem}{Theorem}[section]
\newtheorem{maintheorem}{Theorem}
\newtheorem{corollary}[theorem]{Corollary}
\newtheorem{definition}[theorem]{Definition}
\newtheorem{proposition}[theorem]{Proposition}
\newtheorem{lemma}[theorem]{Lemma}
\theoremstyle{remark}\newtheorem*{convention*}{Convention}
\newtheorem{notation}[theorem]{Notation}
\newtheorem{example}[theorem]{Example}
\newtheorem{remark}[theorem]{Remark}

\hypersetup{
    colorlinks = true,
    linkcolor = {blue},
    citecolor = {magenta},
    urlcolor = {black}
}

\title{Taylor expansions over generalised power series}

\author{Vincent Bagayoko}
\affil{IMJ-PRG}
\author{Vincenzo Mantova}
\affil{University of Leeds}

\begin{document}

\maketitle
\begin{abstract}
  We study the existence of formal Taylor expansions for functions defined on
  fields of generalised series. We prove a general result for the existence
  and convergence of those expansions for fields equipped with a derivation
  and an exponential function, and apply this to the case of standard fields
  of transseries, such as $\log$-$\exp$ transseries and $\omega$-series.
\end{abstract}

\section*{Introduction}

In classical analysis, given a ring of real (or complex) smooth germs at a
point $a \in \mathbb{R}$, computing the Taylor series
\[ f \mapsto f (a) + f' (a) X + \frac{1}{2} f'' (a) X^2 + \cdots \]
yields a differential ring homomorphism into the ring of power series
$\mathbb{R} \llbracket X \rrbracket$. For analytic functions, \tmtextit{by
definition}, this is an \tmtextit{embedding} whose image is contained in the
subring $\mathbb{R} \left\{ \!\!\! \{ X \} \!\!\! \right\}$ of the
\tmtextup{convergent} power series: in this case, the power series of $f$
encodes all the local information about $f$, and by uniqueness of analytic
continuations, also its behaviour on the maximal interval over which $f$
extends to an analytic function.

More general power series can be used to study real analytic functions at
points where they fail to be analytic. \tmtextit{Transseries} in particular
allow writing expansion that include the symbols $\exp$, $\log$ in order to
account for essential singularities of respectively exponential and
logarithmic type. For instance, Stirling's formula for the $\Gamma (x)$
function can be read as the transserial expansion
\[ \Gamma (x) = \sqrt{2 \pi} e^{x \log x - x - \frac{1}{2} \log x} + \frac{B_2
   \sqrt{2 \pi}}{2} e^{x \log x - x - \frac{7}{2} \log x} + \cdots \]

Transseries come in different flavours: grid-based transseries
{\cite{Ec92,vdH:ln}} and log-exp transseries {\cite{DG87,Ec92,vdDMM}} are
contained in the class-sized field $\mathbb{R} \llangle \omega \rrangle$ of
$\omega$-series {\cite{BM19}}, and larger systems of so-called hyperseries
{\cite{vdH:loghyp,BvdHK:hyp}} including transexponential terms, and even
surreal numbers {\cite{BM18,BM19,vdH:hypno}}, fit in that picture. For the
purposes of this note, we deal with transseries in an abstract way. First, let
us fix an algebra $\mathbb{A}= \mathbf{K} \llbracket \mathfrak{M} \rrbracket$
of \tmtextit{Noetherian series} (see \Cref{subsection-Noetherian-series}),
where $\mathfrak{M}$ is a partialy ordered monoid and $\mathbf{K}$ is a field,
both of which can be proper classes. Noetherian series come with a natural
formal notion of infinite sum $\sum_{i \in I} f_i$ whenever the family
$(f_i)_{i \in I}$ is \tmtextit{summable}. We then fix a \tmtextit{derivation}
$\partial$ which is \tmtextit{strongly ($\mathbf{K}$-)linear}, meaning it
commutes with infinite sums and it vanishes on $\mathbf{K}$.

Since $\mathbb{A}$ is a differential ring, the Taylor series of some $f \in
\mathbb{A}$ is defined as
\[ f + f' X + \frac{f''}{2!} X^2 + \cdots = \sum_{n \in \mathbb{N}}
   \frac{f^{(n)}}{n!} X^n . \]
This induces {\cite{Ng:Taylor}} a differential ring homomorphism from
$\mathbb{A}$ to $\mathbb{A} \llbracket X \rrbracket$.

Since $\mathbb{A}$ allows some infinite sums, one may ask for which $\delta
\in \mathbb{A}$ the family $\left( \frac{f^{(k)}}{k!} \delta^k \right)_{k \in
\mathbb{N}}$ is summable, in which case its sum is an element of $\mathbb{A}$.
Given a power series $P = \sum_{k \in \mathbb{N}} P_k X^k \in \mathbb{A}
\llbracket X \rrbracket$ with coefficients in $\mathbb{A}$, let the \textit{convergence locus} of $P$ be
\[ \tmop{Conv} (P) \assign \left\{ \delta \in \mathbb{A} \suchthat (P_k
   \delta^k)_{k \in \mathbb{N}}  \text{ is summable} \right\} . \]
One may easily verify that $\tmop{Conv} (P)$ is always convex,
$0\in \tmop{Conv} (P)$, and for instance $\tmop{Conv} (P + Q) \supseteq
\tmop{Conv} (P) \cap \tmop{Conv} (Q)$ (Lemma~\ref{lem-conv-sum-prod}). We write
$\tmop{Conv} (f)$ for the convergence locus of the Taylor series of $f$.

The main result of this paper are explicit and almost sharp bounds for
$\tmop{Conv} (f)$ when $\mathbb{A}$ is an algebra of transseries. More
precisely, we say that $\mathbb{A}$ is a {\tmem{differential pre-logarithmic
Hahn field}} if $\mathbf{K}$ is an ordered field and we are given a morphism
of ordered monoids $\ell \of (\mathfrak{M}, \cdot, 1, \prec) \rightarrow
(\mathbb{A}, +, 0, <)$, which we call \tmtextit{pre-logarithm}, satisfying the
following assumptions:
\begin{enumerate}
  \item $\ell (\mathfrak{M})$ is closed under \tmtextit{truncation} (see
  Definition~\ref{def-truncation});
  
  \item for all $\mathfrak{m} \in \mathfrak{M}$, $\mathfrak{m}^{\dag} \assign
  \frac{\mathfrak{m}'}{\mathfrak{m}} = (\ell (\mathfrak{m}))'$.
\end{enumerate}
If $\mathbf{K}$ is an ordered exponential field {\cite{Kuhl:ordexp}}, then
$\ell$ extends to a morphism $\log \of \mathbb{A}^{> 0} \rightarrow
\mathbb{A}$ which we call \tmtextit{logarithm} (see
Remark~\ref{rem-transserial-extension}).

Our conditions and results are stated in terms of valuation theory, using the
dominance relations $\preccurlyeq$ and $\prec$
{\cite[Definition~3.1.1]{vdH:mt}} and equivalence relations $\asymp$ and
$\sim$ coming from the standard valuation on $\mathbb{A}$ (see
{\cite[Notations, p96]{vdH:mt}}). We require the following technical condition
(\ref{eq-spec-cond}): there is some $x \in \mathfrak{M}$ such that for any
$\mathfrak{m} \in \mathfrak{M}$ we have
\begin{equation}
  \mathfrak{m}^{\dag} \preccurlyeq x^{- 1} \Rightarrow (\tmop{supp}
  \mathfrak{m}')^{\dag} \preccurlyeq x^{- 1} \text{\quad and\quad}
  \mathfrak{m}^{\dag} \succ x^{- 1} \Rightarrow (\tmop{supp}
  \mathfrak{m}')^{\dag} \asymp \mathfrak{m}^{\dag} .
  \label{eq-spec-cond-intro}
\end{equation}
This condition, as arbitrary it may seem, is satisfied in most contexts of
interest (see \Cref{subsection-application-w-series}). Under this condition,
we obtain:

\begin{maintheorem}[\Cref{th-Taylor-convergence} for $\triangle =
  \tmop{id}_{\mathbb{S}}$]\label{main:conv}
  Let $\mathbb{A}$ be a differential
  pre-logarithmic Hahn field satisfying {\tmem{(\ref{eq-spec-cond-intro})}}.
  Then, for all $f \in \mathbb{A}$, we have
  \[ \tmop{Conv} (f) = \bigcap_{\mathfrak{m} \in \tmop{supp} (f)} \tmop{Conv}
     (\mathfrak{m}) \supseteq \left\{ \delta \in \mathbb{A} \suchthat \delta
     \prec x, \mathfrak{m^{\dag}} \delta \prec 1 \text{ for all } \mathfrak{m}
     \in \tmop{supp} (f)  \right\} . \]
\end{maintheorem}

The bound is often sharp, meaning that $\tmop{Conv} (f)$ is often equal to the
right hand side: if $\mathfrak{m}^{\dag} \delta \succcurlyeq 1$ for some
$\mathfrak{m} \in \tmop{supp} (f)$, or if $\delta \succ x$, then $\delta
\notin \tmop{Conv} (f)$ (see Remark~\ref{rem-radius}). However, there are situations
where the Taylor series of $f$ converges on some $\delta \asymp x$, in which
case the above inclusion becomes strict.

We shall actually prove an even stronger statement. Many fields of
transseries admit a \tmtextit{composition law}: a function that takes $f, g
\in \mathbb{A}$ with $g \in \mathbb{A}^{> \mathbf{K}}$ and
return $f \circ g \in \mathbb{A}$, such that
\begin{enumerate}
  \item the map $f \mapsto f \circ g$ is strongly $\mathbf{K}$-linear;
  
  \item $(\log f) \circ g = \log (f \circ g)$;
  
  \item $(f \circ g) = (f' \circ g) g'$;
  
  \item $(f \circ g) \circ h = f \circ (g \circ h)$.
\end{enumerate}
We are then interested in whether the identity
\[ f \circ (g + \delta) = f \circ g + (f' \circ g) \delta + \frac{f'' \circ
   g}{2!} \delta^2 + \cdots \]
holds, and for which $\delta$. We formalise this notion by fixing an arbitrary
strongly linear operator $\triangle \of \mathbb{A} \rightarrow \mathbb{B}$.
For $P \in \mathbb{A} \llbracket X \rrbracket$, let $\tmop{Conv}_{\triangle}
(P)$ be the convergence locus of the power series $\triangle (P)$ (meaning
that we apply $\triangle$ on each coefficient of $P$) and
$\tmop{Conv}_{\triangle} (f)$ be the same for the Taylor series of $f \in
\mathbb{A}$ in place of $P$. We shall prove the following:

\begin{maintheorem}[\Cref{th-Taylor-convergence}]\label{main:conv-triangle}
  Let $\mathbb{A}$ be a differential pre-logarithmic Hahn field satisfying {\tmem{(\ref{eq-spec-cond-intro})}} and let $\triangle \of \mathbb{A}
  \rightarrow \mathbb{B}$ be strongly linear morphism of algebras. For all $f
  \in \mathbb{A}$, we have
  \begin{eqnarray*}\tmop{Conv}_{\triangle} (f) & =  & \bigcap_{\mathfrak{m} \in \tmop{supp} (f)}
     \tmop{Conv}_{\triangle} (\mathfrak{m}) \\ & \supseteq & \left\{ \varepsilon \in
     \mathbb{B} \suchthat \varepsilon \prec \triangle (x), \triangle
     (\mathfrak{m}^{\dagger}) \varepsilon \prec 1 \text{ for all }
     \mathfrak{m} \in \tmop{supp} (f)  \right\} . \end{eqnarray*}
\end{maintheorem}

\

We also show (Theorems~\ref{th-Taylor-analytic} and~\ref{th-chain-rule}) that
if $\triangle$ commutes with families of analytic functions on $\mathbb{S}$
and $\mathbb{T}$, or if it satisfies a chain rule with respect to a derivation
on $\mathbb{T}$, then so does its ``Taylor deformation'' operator $f \mapsto
\sum_{k \in \mathbb{N}} \frac{\triangle (f^{(k)})}{k!} \delta^k$. We then
apply these results in the case of $\omega$-series, taking $\triangle \of f
\mapsto f \circ g$, and show:

\begin{maintheorem}
  \label{th-main}Let $f \in \mathbb{R} \llangle \omega \rrangle$ and let $g,
  \delta \in \mathbb{R} \llangle \omega \rrangle$ with $g >\mathbb{R}$ and
  $\delta \prec g$. Assume that $(\mathfrak{m}^{\dag} \circ g) \delta \prec 1$
  for all $\mathfrak{m} \in \tmop{supp} f$ (i.e. $\delta \in \tmop{Conv}_{h
  \mapsto h \circ g} (f)$). Then we have
  \[ f \circ (g + \delta) = \sum_{k \in \mathbb{N}} \frac{f^{(k)} \circ g}{k!}
     \delta^k . \]
\end{maintheorem}

The same conclusion applies for other maps $\triangle$. For instance, we could
take $g, \delta$ to be surreal numbers using the composition of {\cite{BM19}}.
The convergence locuses specified in our theorems are optimal (see
Remark~\ref{rem-radius}), and generalise various existing results about Taylor
expansions in fields of transseries. Their history is less linear than one
might think, so we feel it is appropriate to briefly discuss those results,
and their limitations, in chronological order:
\begin{itemizedot}
  \item {\'E}calle {\cite[4.1.26bis]{Ec92}} considered Taylor expansions of
  grid-based transseries or log--exp transseries. His bounds for the
  convergence locus are sometimes too small to be used appropriately (see
  {\cite[(6.32)]{vdDMM01}}).
  
  \item Van den Dries, Macintyre and Marker {\cite[(6.8)-(4)]{vdDMM01}} showed
  that logarithmic-exponential transseries in $\mathbb{T}_{\tmop{LE}}$ have
  Taylor expansions, but gave a non-optimal convergence locus.
  
  \item Schmeling {\cite[Section~6]{Schm01}} showed that $\omega$-series act
  on transserial fields that are closed under exponentiation, and that they
  have Taylor expansions with optimal convergence locus. Unfortunately, his
  proof is incomplete.
  
  \item Van der Hoeven {\cite[Proposition~5.11(c)]{vdH:ln}} showed that the
  theorem above is valid in the subfield of $\mathbb{R} \llangle \omega
  \rrangle$ of grid-based transseries.
  
  \item Berarducci and Mantova defined {\cite[Theorem~6.3]{BM19}} a
  composition law $\mathord{\circ} \of \mathbb{R} \llangle \omega \rrangle
  \times \mathbf{No}^{>\mathbb{R}} \longrightarrow \mathbf{No}$ on Conway's
  ordered field $\mathbf{No}$ of surreal numbers {\cite{Con76}}, and showed
  {\cite[Theorem~7.5]{BM19}} that a series $f \in \mathbb{R} \llangle \omega
  \rrangle$ has a Taylor expansion
  \[ f \circ (\xi + \delta) = \sum_{k \in \mathbb{N}} \frac{f^{(k)} \circ
     \xi}{k!} \delta^k \]
  at every $\xi \in \mathbf{No}^{>\mathbb{R}}$ for small enough (but
  undetermined) $\delta \in \mathbf{No}$ depending on $f$ and $\xi$.
  
  \item Van den Dries, van der Hoeven and Kaplan
  {\cite[Proposition~8.1]{vdH:loghyp}} showed that the theorem above is valid
  in their field of so-called logarithmic hyperseries.
\end{itemizedot}
In particular, there is no proof in the literature of the optimal result for
$\omega$-series or even $\log$-$\exp$ transseries. Our method is designed to
be as general as possible, and we will use it subsequently in order to derive
Taylor expansions results for larger fields of transseries, notably the
hyperexponential closure {\cite{BvdHK:hyp}} of the field of logarithmic
hyperseries, and later, surreal numbers.

We prove the main theorem by switching perspective. Instead of fixing $f$ and
looking at which $\delta$'s make the Taylor series of $f$ around $g$
convergent, we fix $\delta$ and $g$ and study which series in $\mathbb{A}
\llbracket X \rrbracket$ converge around $g$ at least as far as $\: \delta$.
Fixing $(g, \delta) \in \mathbb{A}^{> \mathbf{K}} \times \mathbb{A}$, the
operator
\[ \mathbb{A} \longrightarrow \mathbb{A} \: ; \: f \mapsto \sum_{k \in
   \mathbb{N}} \frac{f^{(k)} \circ g}{k!} \delta^k \]
can be obtained in the following three steps:
\begin{equation}
  \mathbb{A} \longrightarrow \mathbb{A} \llbracket X \rrbracket \: ; \: f
  \mapsto \sum_{k \in \mathbb{N}} \frac{f^{(k)}}{k!} X^k
  \label{eq-intro-diff1},
\end{equation}
\begin{equation}
  \mathbb{A} \llbracket X \rrbracket \longrightarrow \mathbb{A} \llbracket X
  \rrbracket \: ; \: \sum_{k \in \mathbb{N}} P_k X^k \mapsto \sum_{k \in
  \mathbb{N}} (P_k \circ g) X^k \label{eq-intro-comp},
\end{equation}
\begin{equation}
  \mathbb{A} \llbracket X \rrbracket \longrightarrow \mathbb{A} \hspace{0.8em}
  ; \: \sum_{k \in \mathbb{N}} P_k X^k \mapsto \sum_{k \in \mathbb{N}} P_k
  \delta^k, \label{eq-intro-ev}
\end{equation}
each of which defines a strongly linear morphism of algebras. The first step
itself can be seen as an infinite sum of iterated operators
\begin{equation}
  \sum_{k \in \mathbb{N}} \frac{1}{k!}  (X \cdot \overline{\partial})^k
  \label{eq-intro-scarysum}
\end{equation}
evaluated at $a \in \mathbb{K}$, where $\overline{\partial}$ is the operator
\begin{equation}
  \overline{\partial} \of \mathbb{A} \llbracket X \rrbracket \longrightarrow
  \mathbb{A} \llbracket X \rrbracket \: ; \: \sum_{k \in \mathbb{N}} P_k X^k
  \mapsto \sum_{k \in \mathbb{N}} P_k' X^k . \label{eq-intro-diff2}
\end{equation}

As a consequence of van der Hoeven's implicit function
theorem~{\cite{vdH:noeth}}, the summability of (\ref{eq-intro-scarysum}) only
requires the operator $X \cdot \overline{\partial}$ to commute with infinite
sums and to be contracting in a valuation theoretic sense. What makes
convergence of Taylor series difficult is the fact that the operation
(\ref{eq-intro-ev}) is {\tmem{not}} defined, in general, on the whole algebra
$\mathbb{A} \llbracket X \rrbracket$. In order to obtain the largest
subalgebra on which all operations can be performed, we are to find the domain
of (\ref{eq-intro-ev}), then its preimage under (\ref{eq-intro-comp}), and then
the preimage of the latter under (\ref{eq-intro-diff1}). This in turn leads us
to study conditions under which $X \cdot \overline{\partial}$ is contracting
and commutes with sums, and under which the infinite sum
(\ref{eq-intro-scarysum}) ranges in that last domain.

We determine the subalgebra of Noetherian series of $\mathbb{A} \llbracket X
\rrbracket$ on which (\ref{eq-intro-ev}) is defined in
\Cref{subsection-subalgebras}. A large part of the problem then reduces to
finding subalgebras of Noetherian series of $\mathbb{A} \llbracket X
\rrbracket$ between which endomorphisms of $\mathbb{A}$ and derivations on
$\mathbb{A}$ can be extended as in (\ref{eq-intro-comp},
\ref{eq-intro-diff2}). This is the purpose of
Sections~\ref{subsection-cut-extensions} and~\ref{subsection-cut-extensions2}.
We then combine our findings in \Cref{section-Taylor-expansions} to obtain the
main theorems and apply them to $\omega$-series.

\begin{convention*}
  Before we start, we set a few conventions.
  
  \begin{descriptioncompact}
    \item[Set theory] We adopt the set-theoretic framework of
    {\cite{BvdHK:hyp}}. The underlying set theory of this paper is NBG set
    theory with the axiom of limitation of size, a conservative extension of
    ZFC which allows us to prove statements about proper classes.\footnote{All
    of the arguments in this paper also work in ZFC, provided one fixes an
    uncountable regular cardinal $\kappa$ and replaces each occurence of the
    word `set' with `set of size $< \kappa$'.}
    
    \item[Ordinals] We consider the class $\mathbf{On}$ of all ordinals as a
    generalised, regular ordinal. We recall that the
    {\tmem{cofinality}}{\index{cofinality}} of a linearly ordered class
    $(\mathbf{L}, <)$ without maximum is the unique regular generalised
    ordinal $\tmmathbf{\kappa}$ such that there exists a nondecreasing map
    $\tmmathbf{\kappa} \longrightarrow \mathbf{L}$ whose range is cofinal in
    $\mathbf{L}$. Assuming limitation of size, this is always defined.
    
    \item[Ordered monoids] If $(\mathbf{M}, 0, +, <)$ is an ordered monoid,
    then $\mathbf{M}^{>}$\label{autolab1} denotes its subclass of strictly
    positive elements in $\mathbf{M}$, whereas
    $\mathbf{M}^{\neq}$\label{autolab2} denotes the class of non-zero elements
    of $\mathbf{M}$.
  \end{descriptioncompact}
\end{convention*}

\section{Noetherian series}\label{section-Noetherian-series}

\subsection{Noetherian orderings}\label{subsection-Noetherian-orderings}

We first introduce the types of ordered sets that will be the supports of our
formal series throughout the paper.

\begin{definition}
  Let $(\mathbf{X}, <)$ be a partially ordered class. A
  {\tmem{{\tmstrong{chain}}}}{\index{chain}} in $\mathbf{X}$ is a linearly
  ordered subclass of $\mathbf{X}$. A {\tmem{{\tmstrong{decreasing
  chain}}}}{\index{decreasing chain}} in $\mathbf{X}$ is chain $\mathbf{Y}
  \subseteq \mathbf{X}$ without minimal element, i.e. with
  \[ \forall y \in \mathbf{Y}, \exists z \in \mathbf{Y}, (z < y) . \]
  An {\tmstrong{{\tmem{antichain}}}}{\index{antichain}} in $\mathbf{X}$ is a
  subclass $\mathbf{Y} \subseteq \mathbf{X}$, no two distinct elements of
  which are comparable, i.e. with
  \[ \forall y, z \in \mathbf{Y}, y \leqslant z \Longrightarrow y = z. \]
  We say that $(\mathbf{X}, <)$ is
  {\tmstrong{{\tmem{Noetherian}}}}{\index{Noetherian ordering}} if there are
  no infinite decreasing chains and no infinite antichains in $(\mathbf{X},
  <)$.
\end{definition}

Noetherianity is a strengthening of well-foundedness (no decreasing chains),
and a weakening of well-orderedness (the conjunction of linearity and
well-foundedness). Noetherian orderings are sometimes called
well-partial-orderings. It will be convenient to rely on the notion of bad
sequence and minimal bad sequence of {\cite{NW63}}. If $(\mathbf{X},
<_{\mathbf{X}})$ is an ordered class, then a {\tmem{bad sequence}}{\index{bad
sequence}} in $\mathbf{X}$ is a sequence $u \of \mathbb{N} \longrightarrow
\mathbf{X}$ such that there are {\tmem{no}} numbers $i, j \in \mathbb{N}$ with
$i < j$ and $u_i \leqslant_{\mathbf{X}} u_j$.

\begin{lemma}
  \label{lem-Noeth-subsequence}{\tmem{{\cite[Theorem~2.1]{Hig52}}}} Let
  $(\mathbf{X}, <_{\mathbf{X}})$ be a partially ordered class. The following
  statements are equivalent
  \begin{enumeratealpha}
    \item $\label{lem-Noeth-subsequence-a} (\mathbf{X}, <_{\mathbf{X}})$ is
    Noetherian.
    
    \item \label{lem-Noeth-subsequence-b}There is no bad sequence in
    $(\mathbf{X}, <_{\mathbf{X}})$.
    
    \item \label{lem-Noeth-subsequence-c}Every sequence in $\mathbf{X}$ has a
    weakly increasing subsequence.
  \end{enumeratealpha}
\end{lemma}

\subsection{Noetherian subsets in ordered monoids}

An {\tmem{ordered monoid}}{\index{ordered monoid}} is a tuple $(\mathfrak{M},
\cdot, 1, \prec)$ where $(\mathfrak{M}, \cdot, 1)$ is a monoid and $\prec$ is
a partial ordering on $\mathfrak{M}$ with
\begin{equation}
  \forall \mathfrak{u}, \mathfrak{v}, \mathfrak{w} \in \mathfrak{M},
  \mathfrak{u} \prec \mathfrak{v} \Longrightarrow (\mathfrak{u}\mathfrak{w}
  \prec \mathfrak{v}\mathfrak{w} \wedge \mathfrak{w}\mathfrak{u} \prec
  \mathfrak{w}\mathfrak{v}) . \label{eq-ordered-monoid}
\end{equation}
We fix an ordered monoid $(\mathfrak{M}, \cdot, 1, \prec)$. For $\mathfrak{S}
\subseteq \mathfrak{M}$ we write\label{autolab3}
\[ \mathfrak{S}^n \assign \underset{\text{$n$ times}}{\mathfrak{S} \cdots
   \mathfrak{S}} = \{ \mathfrak{s}_1 \cdots \mathfrak{s}_n \suchthat
   \mathfrak{s}_1, \ldots, \mathfrak{s}_n \in \mathfrak{S} \} \]
and
\[ \mathfrak{S}^{\infty} \assign \bigcup_{n \in \mathbb{N}} \mathfrak{S}^n =
   \{ \mathfrak{s}_1 \cdots \mathfrak{s}_n \suchthat n \in \mathbb{N} \wedge
   \mathfrak{s}_1, \ldots, \mathfrak{s}_n \in \mathfrak{S} \} . \]
As in {\cite{Hig52}}, as consequences of {\cite[Theorems~2.3 and 4.3]{Hig52}},
we obtain

\begin{lemma}
  \label{lem-Neumann-easy}{\tmem{{\cite{Hig52}}}} Let $\mathfrak{S},
  \mathfrak{T} \subseteq \mathfrak{M}$ be Noetherian. Then the class
  $\mathfrak{S} \cdot \mathfrak{T}$ is Noetherian. Moreover, for all
  $\mathfrak{m} \in \mathfrak{S} \cdot \mathfrak{T}$, the set $\{
  (\mathfrak{u}, \mathfrak{v}) \in \mathfrak{S} \times \mathfrak{T} \suchthat
  \mathfrak{m}=\mathfrak{u}\mathfrak{v} \}$ is finite.
\end{lemma}

We say that $\mathfrak{S} \subseteq \mathfrak{M}$ is Noetherian in
$(\mathfrak{M}, \succ)$ or that it is a Noetherian subclass of $(\mathfrak{M},
\succ)$ if it is Noetherian for the reverse ordering on $\mathfrak{M}$.

\begin{proposition}
  \label{prop-Neumann}Let $\mathfrak{S} \subseteq \mathfrak{M}^{\prec}$ be a
  Noetherian subclass of $(\mathfrak{M}, \succ)$. Then the class
  $\mathfrak{S}^{\infty}$ is a Noetherian subset of $(\mathfrak{M}, \succ)$.
  Moreover, for all $\mathfrak{m} \in \mathfrak{S}^{\infty}$, the set $\{ n
  \in \mathbb{N} \suchthat \mathfrak{m} \in \mathfrak{S}^n \}$ is finite.
\end{proposition}

\subsection{Algebras of Noetherian series}\label{subsection-Noetherian-series}

For the rest of \Cref{section-Noetherian-series}, we fix a field $\mathbf{K}$.
Let $(\mathfrak{M}, \cdot, 1, \prec)$ be an ordered monoid. We write
$\mathbf{K} \llbracket \mathfrak{M} \rrbracket$\label{autolab4} for the class
of functions $s \of \mathfrak{M} \longrightarrow \mathbf{K}$ whose support
\[ \tmop{supp} s \assign \{ \mathfrak{m} \in \mathfrak{M} \suchthat s
   (\mathfrak{m}) \neq 0 \} \text{\label{autolab5}} \]
is a Noetherian sub{\tmem{set}} of $(\mathfrak{M}, \succ)$. For $s \in
\mathbf{K} \llbracket \mathfrak{M} \rrbracket$, we write $\max \tmop{supp} s$
for the (finite) set of maximal elements in $\tmop{supp} s$. Since Noetherian
subsets of $(\mathfrak{M}, \succ)$ are closed under binary unions, this class
is a subspace of the vector space of functions $\mathfrak{M} \longrightarrow
\mathbf{K}$ with set-sized support.

For $s, t \in \mathbf{K} \llbracket \mathfrak{M} \rrbracket$, we define a Cauchy
product $(s \cdot t) \of \mathfrak{M} \longrightarrow \mathbf{K}$ by
\begin{equation}
  \forall \mathfrak{m} \in \mathfrak{M}, (s \cdot t) (\mathfrak{m}) \assign
  \sum_{\mathfrak{u}, \mathfrak{v} \in \mathfrak{M} \wedge
  \mathfrak{u}\mathfrak{v}=\mathfrak{m}} s (\mathfrak{u}) t (\mathfrak{v}) .
  \label{eq-Cauchy-prod}
\end{equation}
In view of Lemma~\ref{lem-Neumann-easy}, each sum in (\ref{eq-Cauchy-prod}) has
finite support, and so is a well-defined element of $\tmmathbf{\Kappa}$.
Moreover $\tmop{supp} (s \cdot t) \subseteq (\tmop{supp} s) \cdot (\tmop{supp}
t)$ is a Noetherian subset of $(\mathfrak{M}, \succ)$, so $s \cdot t$ is a
well-defined element of~$\mathbf{K} \llbracket \mathfrak{M} \rrbracket$.

Writing $\mathbb{1}_{\mathfrak{S}}$ for the indicator function $\mathfrak{M}
\longrightarrow \{0, 1\} \subseteq \mathbf{K}$ of a subclass $\mathfrak{S}
\subseteq \mathfrak{M}$, we have an embedding of ordered monoids
\[ (\mathfrak{M}, \cdot, 1) \longrightarrow (\mathbb{A} \setminus \{0\},
   \cdot, 1) ; \mathfrak{m} \mapsto \mathbb{1}_{\{\mathfrak{m}\}} \]
We identify $\mathfrak{M}$ with its image in $\mathbb{A} \setminus \{0\}$, and
likewise identify $\tmmathbf{\Kappa}$ with the image of the field emdedding $c
\mapsto c 1$. The elements of $\mathfrak{M}$ are called
{\tmem{monomials}}{\index{monomial}}, whereas those in $\mathbf{K}^{\times}
\mathfrak{M}$ are called {\tmem{terms}}{\index{term}}. The structure
$\mathbf{K} \llbracket \mathfrak{M} \rrbracket$ is a unital algebra over
$\mathbf{K}$ (see {\cite[Proposition~3.9]{BKKMP:strong}}). We call $\mathbf{K}
\llbracket \mathfrak{M} \rrbracket$ the algebra of Noetherian series over
$\mathbf{K}$ with monomials in $\mathfrak{M}$. It has the same characteristic
as $\mathbf{K}$.

\begin{remark}
  Higman {\cite[Section~5]{Hig52}} considered the same structures in the case
  when $(\mathfrak{M}, \cdot, 1)$ is cancellative, which is not necessary in
  the proofs above. If one imposes that $\mathfrak{M}$ is a linearly ordered
  group, then $\mathbf{K} \llbracket \mathfrak{M} \rrbracket$ is a skew field
  (see {\cite[Chapter~2]{Cohn:skew}}).
\end{remark}

\begin{remark}
  If $\mathfrak{M}$ is trivial, then the embedding $\mathbf{K} \longrightarrow
  \mathbf{K} \llbracket \mathfrak{M} \rrbracket \: ; \: c \mapsto c 1$ is an
  isomorphism.
\end{remark}

\begin{example}
  Taking $\mathfrak{M}$ to be a multiplicative copy $X^{\mathbb{N}^n}$ of the
  partially ordered monoid $(\mathbb{N}^n, +, 0, >)$ (i.e. the $n$-th power of
  the linearly ordered monoid $(\mathbb{N}, +, 0, >)$), we obtain the algebra
  $\mathbf{K} \llbracket X^{\mathbb{N}^n} \rrbracket \simeq \mathbf{K}
  \llbracket X_1, \ldots, X_n \rrbracket$ of formal power series in $n$
  commmuting variables over~$\mathbf{K}$.
  
  Taking $\mathfrak{M}$ to be a multiplicative copy $X^{\mathbb{Z}}$ of
  $(\mathbb{Z}, +, 0, >)$, we obtain the field $\mathbf{K} \llbracket
  X^{\mathbb{Z}} \rrbracket$ of formal Laurent series over $\mathbf{K}$.
  
  Taking $\mathfrak{M}$ to be a multiplicative copy $X^{\mathbb{N}^n}$ of the
  partially (and vacuously) ordered monoid $(\mathbb{N}^n, +, 0,
  \varnothing)$, one obtains the algebra $\mathbf{K} [X_1, \ldots, X_n]$ of
  polynomials in $n$ variables over $\mathbf{K}$.
  
  Taking $\mathfrak{M}$ to be the monoid under concatenation of finite words
  over a well-ordered set $(I, <)$, ordered lexicographically, we obtain the
  algebra $\mathbf{K} \llangle I \rrangle$ of formal power series over
  $\mathbf{K}$ in a set of non-commuting indeterminates indexed by $I$.
\end{example}

\subsection{Dominance relation and valuation}

Let $\mathbb{A}= \textbf{{\tmstrong{K}}} \llbracket \mathfrak{M} \rrbracket$
be an algebra of Noetherian series. Given $s, t \in \mathbb{A}$, we
write $s \preccurlyeq t$ if for all $\mathfrak{m} \in \tmop{supp} s$, there is an $\mathfrak{n} \in
\tmop{supp} t$ with $\mathfrak{m} \preccurlyeq \mathfrak{n}$. Then
$\preccurlyeq$ is a linear quasi-ordering on $\mathbb{A}$. We write $\prec$
for the corresponding (strict) ordering $s \prec t \Longleftrightarrow (s
\preccurlyeq t \wedge t \npreccurlyeq s)$. We have $s \prec t$ if and only if
$t \neq 0$, and for all $\mathfrak{m} \in \tmop{supp} s$, there is an
$\mathfrak{n} \in \tmop{supp} t$ with $\mathfrak{m} \prec \mathfrak{n}$. Note
that the inclusion $\mathfrak{M} \subseteq \mathbb{A} \setminus \{0\}$
preserves the orderings on $(\mathfrak{M}, \prec)$ and $(\mathbb{A} \setminus
\{0\}, \prec)$ respectively. We define\label{autolab7} \label{autolab8}
\label{autolab9}
\begin{eqnarray*}
  \mathbb{A}^{\preccurlyeq} & \assign & \{ s \in \mathbb{A} \suchthat
  \tmop{supp} s \preccurlyeq 1 \} = \{s \in \mathbb{A} \suchthat s
  \preccurlyeq 1\},\\
  \mathbb{A}^{\prec} & \assign & \{ s \in \mathbb{A} \suchthat \tmop{supp} s
  \prec 1 \} = \{ s \in \mathbb{A} \suchthat s \prec 1 \} \text{, \quad and}\\
  \mathbb{A}^{\prec s} & \assign & \{t \in \mathbb{A} \suchthat t \prec s\}
\end{eqnarray*}
for all $s \in \mathbb{A}$. Series in $\mathbb{A}^{\prec}$ are said
{\tmem{infinitesimal}}{\index{infinitesimal series}}, whereas series in
$\mathbb{A}^{\preccurlyeq}$ are said {\tmem{bounded}}{\index{bounded series}}.
Note that $\mathbb{A}^{\preccurlyeq} = \mathbf{K} \oplus \mathbb{A}^{\prec}$.
The algebra $\mathbb{A}^{\preccurlyeq}$ is always local with maximal ideal
$\mathbb{A}^{\prec}$ ({\cite[Proposition~2.8]{BKKMP:strong}}).

Suppose that $(\mathfrak{M}, \prec)$ is a linearly ordered Abelian group. Then
it is well-known {\cite{Hahn1907}} that $\mathbb{A}$ is a field and
{\cite{Krull32}} that $\mathbb{A}^{\preccurlyeq}$ is a valuation ring of
$\mathbb{A}$. In that case, for all $s, t \in \mathbb{A}$, we have $s
\preccurlyeq t$\label{autolab10} if and only if $t \prec s$ is false. We write
write $s \asymp t$\label{autolab11} if $s \preccurlyeq t$ and $t \preccurlyeq
s$. Then $\preccurlyeq$ is a dominance relation as per
{\cite[Definition~3.1.1]{vdH:mt}}. For $s \in \mathbb{A}^{\times}$, so
$\tmop{supp} s \neq \varnothing$, we write\label{autolab12} \label{autolab13}
\begin{eqnarray*}
  \mathfrak{d}_s & \assign & \max \tmop{supp} s \in \mathfrak{M},\\
  \tau_s & \assign & s (\mathfrak{d}_s) \mathfrak{d}_s \in \mathbf{K}^{\times}
  \mathfrak{M}, \text{\quad and}\\
  c_s & \assign & s (\mathfrak{d}_s) \in \mathbf{K}^{\times} .
\end{eqnarray*}
respectively for the {\tmem{dominant monomial}}{\index{dominant monomial}},
{\tmem{dominant term}}{\index{dominant term}} and {\tmem{leading
coefficient}}{\index{leading coefficient}} of $s$. When $s, t$ are non-zero,
we have $s \prec t$ (resp. $s \preccurlyeq t$, resp. $s \asymp t$) if and only
if $\mathfrak{d}_s \prec \mathfrak{d}_t$ ({\tmabbr{resp.}} $\mathfrak{d}_s
\preccurlyeq \mathfrak{d}_t$, {\tmabbr{resp.}} $\mathfrak{d}_s
=\mathfrak{d}_t$). Moreover, any $s \in \mathbb{A}^{\times}$ can be written
uniquely as
\begin{equation}
  s = c_s \mathfrak{d}_s  (1 + \varepsilon_s) \label{eq-mult-dec}
\end{equation}
where $\mathfrak{d}_s \in \mathfrak{M}$, $c_s \in \mathbf{K}^{\times}$ and
$\varepsilon_s \prec 1$.

Assume finally that $\mathbf{K}$ is an ordered field, while $(\mathfrak{M},
\prec)$ is still a linearly ordered Abelian group. Then we have a positive
cone
\[ \text{$\mathbb{A}^{>} \assign \{ s \in \mathbb{A} \suchthat s \neq 0 \wedge
   c_s > 0 \}$.} \]
on $\mathbb{A}$, that is, defining $s < t \Longleftrightarrow t - s \in
\mathbb{A}^{>}$, the structure $(\mathbb{A}, +, \times, <, \prec)$ is an
ordered valued field with convex valuation ring~$\mathbb{A}^{\preccurlyeq}$.
We write\label{autolab14}
\[ \mathbb{A}^{>, \succ} \assign \{ s \in \mathbb{A} \suchthat s > \mathbf{K}
   \} = \{ s \in \mathbb{A} \suchthat s \geqslant 0 \wedge s \succ 1 \} \]
Series in $\mathbb{A}^{>, \succ}$ are called {\tmem{positive
infinite}}.{\index{purely large series}}{\index{infinitesimal
series}}{\index{positive infinite series}}

\begin{remark}
  \label{rem-alg-closed}If $\mathbf{K}$ is algebraically closed and
  $\mathfrak{M}$ is divisible, then {\cite{MacLane:pow}} the field
  $\mathbb{A}$ is algebraically closed. If $\mathbf{K}$ is real-closed and
  $\mathfrak{M}$ is divisible, then the field $\mathbb{A}$ is real-closed.
\end{remark}

\subsection{Summable families}

We fix an algebra $\mathbb{A}= \mathbf{K} \llbracket \mathfrak{M} \rrbracket$
of Noetherian series.

\begin{definition}
  Let $\mathbf{I}$ be a class. A family $(s_i)_{i \in \mathbf{I}}$ in
  $\mathbb{A}$ is said {\tmem{summable}}{\index{well-based family}} if
  \begin{enumerateroman}
    \item \label{def-summable-i}$\bigcup_{i \in \mathbf{I}} \tmop{supp} s_i$
    is a Noetherian subset of $(\mathfrak{M}, \succ)$, and
    
    \item $\{ i \in \mathbf{I} \suchthat \mathfrak{m} \in \tmop{supp} s_i \}$
    is finite for all $\mathfrak{m} \in \mathfrak{M}$.
  \end{enumerateroman}
  Then we may define the sum $\sum_{i \in \mathbf{I}} s_i$ of $(s_i)_{i \in
  \mathbf{I}}$ as the series
  \[ \sum_{i \in \mathbf{I}} s_i \assign \mathfrak{m} \longmapsto \sum_{i \in
     I} s_i (\mathfrak{m}) \in \mathbb{A}. \]
\end{definition}

If only \ref{def-summable-i} holds, then we say that $(s_i)_{i \in
\mathbf{I}}$ is {\tmem{weakly summable}}. For $s \in \mathbb{A}$, the family
of terms $(s (\mathfrak{m})\mathfrak{m})_{\mathfrak{m} \in \mathfrak{M}}$ is
summable with sum
\[ \sum_{\mathfrak{m} \in \mathfrak{M}} s (\mathfrak{m}) \mathfrak{m}= s. \]
\begin{definition}
  \label{def-truncation}A {\tmstrong{{\tmem{truncation}}}} of $s$ is a series
  of the form $\sum_{\mathfrak{m} \in \mathfrak{I}} s (\mathfrak{m})
  \mathfrak{m}$ where $\mathfrak{I}$ is a subclass of $\mathfrak{M}$ which is
  initial for the ordering $\prec$. A subclass $\mathbf{C}$ of $\mathbb{A}$ is
  said {\tmem{{\tmstrong{closed under truncation}}}} if any truncation of an
  element in $\mathbf{C}$ lies in $\mathbf{C}$.
\end{definition}

As a consequence of Lemma~\ref{lem-Noeth-subsequence}, we obtain:

\begin{lemma}
  \label{lem-summable-bad-sequence}Let $\mathbf{I}$ be a class and let
  $(s_i)_{i \in \mathbf{I}}$ be a family in $\mathbb{A}$. Then $(s_i)_{i \in
  \mathbf{I}}$ is summable (resp. weakly summable) if and only if for each
  injective sequence (resp. sequence) $i \of \mathbb{N} \longrightarrow
  \mathbf{I}$ and each sequence $(\mathfrak{m}_k)_{k \in \mathbb{N}}$ with
  $\mathfrak{m}_k \in \tmop{supp} s_{i (k)}$ for all $k \in \mathbb{N}$, there
  are a $k, l \in \mathbb{N}$ with $k < l$ and $\mathfrak{m}_k \succ
  \mathfrak{m}_l$ (resp. $\mathfrak{m}_k \succcurlyeq \mathfrak{m}_l$).
\end{lemma}

\begin{proposition}
  \label{prop-summation-by-parts}{\tmem{{\cite[Proposition~3.1(e)]{vdH:noeth}}}}
  Let $\mathbf{I}, \mathbf{J}$ be classes and let $(\mathbf{I}_j)_{j \in
  \mathbf{J}}$ be a family of subclasses of $\mathbf{I}$ such that
  $\mathbf{I}$ is the disjoint union $\mathbf{I} = \bigsqcup_{j \in
  \mathbf{J}} \mathbf{I}_j$. Let $(s_i)_{i \in \mathbf{I}}$ be a summable
  family. Then for all $j \in \mathbf{J}$, the family $(s_i)_{i \in
  \mathbf{I}_j}$ is summable, the family $(\sum_{i \in \mathbf{I}_j} s_i)_{j
  \in \mathbf{J}}$ is summable, and
  \[ \sum_{j \in \mathbf{J}} \sum_{i \in \mathbf{I}_j} s_i = \sum_{i \in
     \mathbf{I}} s_i . \]
\end{proposition}

We have the following corollary:

\begin{lemma}
  \label{lem-Fubini}Let $\mathbf{I}, \mathbf{J}$ be classes and let $(s_{i,
  j})_{(i, j) \in \mathbf{I} \times \mathbf{J}}$ be a summable family in
  $\mathbb{A}$. For each $i_0 \in \mathbf{I}$ and for each $j_0 \in
  \mathbf{J}$, the families $(s_{i_0, j})_{j \in \mathbf{J}}$ and $(s_{i,
  j_0})_{i \in \mathbf{I}}$ are summable. Moreover families $\left( \sum_{j
  \in \mathbf{J}} s_{i, j} \right)_{i \in \mathbf{I}}$ and $\left( \sum_{i \in
  \mathbf{I}} s_{i, j} \right)_{j \in \mathbf{J}}$ are summable, with
  \[ \sum_{i \in \mathbf{I}} \left( \sum_{j \in \mathbf{J}} s_{i, j} \right)
     = \sum_{(i, j) \in \mathbf{I} \times \mathbf{J}} s_{i, j} = \sum_{j \in
     \mathbf{J}} \left( \sum_{i \in \mathbf{I}} s_{i, j} \right) . \]
\end{lemma}

We leave it to the reader to check that the sum of two summable families is
summable:

\begin{lemma}
  \label{lem-sum-sum}Let $\mathbf{I}$ be a class and let $(s_i)_{i \in
  \mathbf{I}}$ and $(t_i)_{i \in \mathbf{I}}$ be summable families in
  $\mathbb{A}$ and let $c \in \mathbf{K}$. The family $(s_i + ct_i)_{i \in
  \mathbf{I}}$ is summable with $\sum_{i \in \mathbf{I}} (s_i + ct_i) =
  \sum_{i \in \mathbf{I}} s_i + c \sum_{i \in \mathbf{I}} t_i$.
\end{lemma}

\subsection{Products of summable families}

Let $\mathbb{A}= \mathbf{K} \llbracket \mathfrak{M} \rrbracket$ be a fixed
algebra of Noetherian series. As a consequence of Proposition~\ref{prop-Neumann}, we
obtain:

\begin{lemma}
  \label{lem-Neumann-series}Let $\varepsilon \in \mathbb{A}^{\prec}$ and
  $(c_k)_{k \in \mathbb{N}} \in \mathbf{K}^{\mathbb{N}}$. Then the family
  $(c_k \varepsilon^k)_{k \in \mathbb{N}}$ is summable.
\end{lemma}

\begin{proposition}
  \label{prop-sum-prod}Let $\mathbf{I}, \mathbf{J}$ be classes, and let
  $(s_i)_{i \in \mathbf{I}}$ and $(t_j)_{j \in \mathbf{J}}$ be summable
  families in $\mathbb{A}$. Then $(s_i \cdot t_j)_{(i, j) \in \mathbf{I}
  \times \mathbf{J}}$ is summable, with
  \[ \sum_{(i, j) \in \mathbf{I} \times \mathbf{J}} s_i \cdot t_j = \left(
     \sum_{i \in \mathbf{I}} s_i \right) \cdot \left( \sum_{j \in \mathbf{J}}
     t_j \right) . \]
\end{proposition}

\begin{proof}
  The proof is the same as {\cite[Proposition~3.3]{vdH:noeth}} where the
  commutativity of the monoid does not play a role.
\end{proof}

\begin{proposition}
  \label{prop-sum-weaksum}Let $(s_i)_{i \in \mathbf{I}}$ be a summable family
  and let $(t_i)_{i \in \mathbf{I}}$ be a weakly summable family. Then the
  family $(s_i \cdot t_i)_{i \in \mathbf{I}}$ is summable.
\end{proposition}

\begin{proof}
  Set $\mathfrak{S} \assign \bigcup_{i \in \mathbf{I}} \tmop{supp} s_i$ and
  $\mathfrak{T} \assign \bigcup_{i \in \mathbf{I}} \tmop{supp} t_i$. The class
  $\bigcup_{i \in \mathbf{I}} (\tmop{supp} s_i) \cdot (\tmop{supp} t_i)
  \subseteq \mathfrak{S} \cdot \mathfrak{T}$ is Noetherian by
  Lemma~\ref{lem-Neumann-easy}. For $\mathfrak{m} \in \mathfrak{M}$, the classes
  $\mathfrak{S}_0 \assign \{ \mathfrak{s} \in \mathfrak{S} \suchthat \exists
  \mathfrak{t} \in \mathfrak{T}, \mathfrak{m}=\mathfrak{s}\mathfrak{t} \}$ and
  $\mathfrak{T}_0 = \{ \mathfrak{t} \in \mathfrak{T} \suchthat \exists
  \mathfrak{s} \in \mathfrak{S}, \mathfrak{m}=\mathfrak{s}\mathfrak{t} \}$ are
  both finite. Thus the class $\mathfrak{X} \assign \{ i \in \mathbf{I}
  \suchthat \exists \mathfrak{s} \in \mathfrak{S}_0, \mathfrak{s} \in
  \tmop{supp} s_i \}$ is finite by summability of $(s_i)_{i \in \mathbf{I}}$.
  We deduce that $\{ i \in \mathbf{I} \suchthat \mathfrak{m} \in \tmop{supp}
  s_i t_i \} \subseteq \mathfrak{X}$ is finite. Therefore $(s_i t_i)_{i \in
  \mathbf{I}}$ is summable.
\end{proof}

\begin{proposition}
  \label{prop-adding-finite-factors}Let $\mathbf{I}$ be a class and let $f \of
  \mathbf{I} \longrightarrow \mathbb{N}$ be an arbitrary function. Let
  $(s_i)_{i \in \mathbf{I}}$ be a summable family in $\mathbb{A}$ and let
  $\delta \in \mathbb{A}^{\preccurlyeq}$. The family $(s_i \cdot \delta^{f
  (i)})_{i \in \mathbf{I}}$ is summable.
\end{proposition}

\begin{proof}
  The family $(\delta^{f (i)})_{i \in \mathbf{I}}$ is weakly summable by
  Proposition~\ref{prop-Neumann}, so this follows from Proposition~\ref{prop-sum-weaksum}.
\end{proof}

\begin{lemma}
  \label{lem-well-based-fin-family}Let $(s_i)_{i \in \mathbf{I}}$ be a family
  in $\mathbb{A}$. Assume that there are Noetherian subclasses $\mathfrak{S}$
  and $\mathfrak{T}$ of $(\mathfrak{M}, \succ)$ with $\mathfrak{T} \prec 1$
  and a function $f \of \mathbf{I} \longrightarrow \mathbb{N}$ such that for
  all $i \in \mathbf{I}$, we have
  \[ \tmop{supp} s_i \subseteq \mathfrak{T}^{f (i)} \cdot \mathfrak{S}. \]
  If $(s_j)_{j \in \mathbf{J}}$ is summable whenever $\mathbf{J} \subseteq
  \mathbf{I}$ and $f (\mathbf{J})$ is finite, then $(s_i)_{i \in \mathbf{I}}$
  is summable.
\end{lemma}

\begin{proof}
  Assume for contradiction that $(s_i)_{i \in \mathbf{I}}$ is not summable. So
  there is an injective sequence $(i_k)_{k \in \mathbb{N}} \in
  \mathbf{I}^{\mathbb{N}}$ and a sequence $(\mathfrak{m}_k)_{k \in \mathbb{N}}
  \in \mathfrak{M}^{\mathbb{N}}$ with $\mathfrak{m}_0 \nsucc \mathfrak{m}_1
  \nsucc \cdots$ and $\mathfrak{m}_k \in \tmop{supp} s_{i_k}$ for all $k \in
  \mathbb{N}$. We have $\{ \mathfrak{m}_k \suchthat k \in \mathbb{N} \}
  \subseteq \mathfrak{T}^{\infty} \cdot \mathfrak{S}$ where
  $\mathfrak{T}^{\infty} \cdot \mathfrak{S}$ is Noetherian in $(\mathfrak{M},
  \succ)$ by Lemma~\ref{lem-Neumann-easy} and Proposition~\ref{prop-Neumann}. So $\{
  \mathfrak{m}_k \suchthat k \in \mathbb{N} \}$ is Noetherian and we may
  assume that $(\mathfrak{m}_k)_{k \in \mathbb{N}}$ is constant. Fix
  $\mathfrak{t} \in \mathfrak{T}^{\infty}$ and $\mathfrak{s} \in S$ with
  $\mathfrak{m}_k =\mathfrak{t}\mathfrak{s}$ for all $k \in \mathbb{N}$. We
  have $\mathfrak{t} \in \mathfrak{T}^{f (i_k)}$ for all $k \in \mathbb{N}$.
  By Proposition~\ref{prop-Neumann}, this implies that $\{ f (i_k) \suchthat k \in
  \mathbb{N} \}$ is finite, so $(s_{i_k})_{k \in \mathbb{N}}$ is summable: a
  contradiction.
\end{proof}

\subsection{Strongly linear functions}

Let $\mathbb{A}= \mathbf{K} \llbracket \mathfrak{M} \rrbracket$ and
$\mathbb{B}= \mathbf{K} \llbracket \mathfrak{N} \rrbracket$ be algebras of
Noetherian series over $\mathbf{K}$. Consider a function $\Phi \of \mathbb{A}
\longrightarrow \mathbb{B}$ which is $\mathbf{K}$-linear. Then $\Phi$ is said
{\tmem{strongly linear}}{\index{strongly linear function}} if for every
summable family $(s_i)_{i \in \mathbf{I}}$ in $\mathbb{A}$, the family $(\Phi
(s_i))_{i \in \mathbf{I}}$ in $\mathbb{B}$ is summable, with
\[ \Phi \left( \sum_{i \in \mathbf{I}} s_i \right) = \sum_{i \in \mathbf{I}}
   \Phi (s_i) . \]
\begin{definition}
  A function $\Phi \of \mathfrak{M} \longrightarrow \mathbb{B}$ is said
  {\tmstrong{{\tmem{Noetherian}}}}{\index{Noetherian function}} if for all
  Noetherian subsets $\mathfrak{S}$ of $(\mathfrak{M}, \succ)$, the family
  $(\Phi (\mathfrak{m}))_{\mathfrak{m} \in \mathfrak{S}}$ is summable in
  $\mathbb{B}$.
\end{definition}

\begin{proposition}
  \label{prop-strongly-linear-extension}{\tmem{{\cite[Proposition~3.5]{vdH:noeth}}}}
  Assume that $\Phi \of \mathfrak{M} \longrightarrow \mathbb{B}$ is
  Noetherian. Then $\Phi$ extends uniquely into a strongly linear function
  $\hat{\Phi} \of \mathbb{A} \longrightarrow \mathbb{B}$. Furthermore, if
  $\Phi$ is a morphism of monoids, then $\hat{\Phi}$ is a morphism of
  algebras.
\end{proposition}

It follows that a linear function $\Phi \of \mathbb{A} \longrightarrow
\mathbb{B}$ is strongly linear if and only if $\Phi \upharpoonleft
\mathfrak{M}$ is Noetherian and $\Phi (s) = \sum_{\mathfrak{m} \in
\mathfrak{M}} s (\mathfrak{m}) \Phi (\mathfrak{m})$ for all $s \in
\mathbb{A}$.

\begin{corollary}
  \label{cor-Noetherian-subgroup}Any embedding of ordered monoids $f \of
  \mathfrak{M} \longrightarrow \mathfrak{N}$ extends uniquely to a strongly
  linear embedding of algebras~$\mathbb{A} \longrightarrow \mathbb{B}$.
\end{corollary}

\begin{lemma}
  \label{lem-Noetherian-function}Let $\Phi \of \mathbb{A} \longrightarrow
  \mathbb{B}$ be strongly linear. Let $(s_i)_{i \in \mathbf{I}}$ be a weakly
  summable family in $\mathbb{A}$. Then $(\Phi (s_i))_{i \in \mathbf{I}}$ is
  weakly summable.
\end{lemma}

\begin{proof}
  Write $\mathfrak{S} \assign \bigcup_{i \in \mathbf{I}} \tmop{supp} s_i$. The
  family $(\mathfrak{s})_{\mathfrak{s} \in \mathfrak{S}}$ is summable, so
  $(\Phi (\mathfrak{s}))_{\mathfrak{s} \in \mathfrak{S}}$ is summable. So the
  class $\mathfrak{T} \assign \bigcup_{\mathfrak{s} \in \mathfrak{S}}
  \tmop{supp} \Phi (\mathfrak{s})$ is Noetherian. But for each $i \in
  \mathbf{I}$, we have $\tmop{supp} \Phi (s_i) = \tmop{supp}
  \sum_{\mathfrak{m} \in \mathfrak{S}} s (\mathfrak{m}) \Phi (\mathfrak{m})
  \subseteq \mathfrak{T}$, so $\bigcup_{i \in \mathbf{I}} \tmop{supp} \Phi
  (s_i) \subseteq \mathfrak{T}$ is Noetherian, i.e. $(\Phi (s_i))_{i \in
  \mathbf{I}}$ is weakly summable.
\end{proof}

\begin{notation}
  Given a function $\Psi \of \mathbf{X} \longrightarrow \mathbf{X}$ on a class
  $\mathbf{X}$ and a $k \in \mathbb{N}$, we write
  $\Psi^{[k]}$\label{autolab15} for the $k$-fold iterate of $\Psi$. So
  $\Psi^{[k]}$ is the function $\mathbf{X} \longrightarrow \mathbf{X}$ with
  $\Psi^{[0]} = \Psi$ and $\Psi^{[k + 1]} \assign \Psi^{[k]} \circ \Psi = \Psi
  \circ \Psi^{[k]}$ for all $k \in \mathbb{N}$.
\end{notation}

\begin{proposition}[Corollary of {\cite[Theorem~6.2]{vdH:noeth}}]\label{prop-iterating-operator}
  Let $\mathbb{A}= \mathbf{K} \llbracket
  \mathfrak{M} \rrbracket$ be an algebra of Noetherian series and let $\Phi
  \of \mathbb{A} \longrightarrow \mathbb{A}$ be strongly linear with $\Phi
  (\mathfrak{m}) \prec \mathfrak{m}$ for all $\mathfrak{m} \in \mathfrak{M}$.
  Let $(c_k)_{k \in \mathbb{N}} \in \mathbf{K}^{\mathbb{N}}$. Then for all $s
  \in \mathbb{A}$, the family $(c_k \Phi^{[k]} (s))_{k \in \mathbb{N}}$ is
  summable, and the function
  \begin{eqnarray*}
    \sum_{k \in \mathbb{N}} c_k \Phi^{[k]} \of \mathbb{A} & \longrightarrow &
    \mathbb{A}\\
    s & \longmapsto & \sum_{k \in \mathbb{N}} c_k \Phi^{[k]} (s)
  \end{eqnarray*}
  is strongly linear.
\end{proposition}

\begin{proof}
  See {\cite[Theorem~1.3 and Corollary~1.4]{vdH:dagap}} and apply
  Proposition~\ref{prop-sum-weaksum} for $(c_k)_{k \in \mathbb{N}}$ and $(\Phi^{[k]}
  (s))_{k \in \mathbb{N}}$ for each $s \in \mathbb{A}$.
\end{proof}

\section{Power series}\label{section-power-series}

\subsection{Elementary analysis on valued fields}

Let $(\mathbf{F}_0, v_0), (\mathbf{F}_1, v_1)$ be (possibly class-sized)
valued fields with non-trivial valuations. For $x_0, \rho_0 \in \mathbf{F}_0$
and $x_1, \rho_1 \in \mathbf{F}_1$ we write
\begin{eqnarray*}
  B_0 (x_0, \rho_0) & \assign & \{y \in \mathbf{F}_0 \suchthat v_0 (y - x_0)
  \geqslant v (\rho_0)\} \text{\quad and}\\
  B_1 (x_1, \rho_1) & \assign & \{y \in \mathbf{F}_1 \suchthat v_1 (y - x_1)
  \geqslant v (\rho_1)\} .
\end{eqnarray*}
Then $\mathbf{F}_0$ and $\mathbf{F}_1$ have a natural topology called the
{\tmem{valuation topology}}. We say that a subclass $\mathbf{X} \subseteq
\mathbf{F}_0$ is a {\tmem{neighborhood}}{\index{neighborhood}} of $x \in
\mathbf{X}$ if there is a $\rho \in \mathbf{F}_0^{\times}$ with $B_0 (x, \rho)
\subseteq \mathbf{X}$. We say that $\mathbf{X}$ is {\tmem{open}}{\index{open}}
if it is empty or if it is a neighborhood of each of its points.

The standard definition of differentiable real-valued function can be
formulated for functions between $\mathbf{F}_0$ and $\mathbf{F}_1$.

\begin{definition}
  Let $x \in \mathbf{F}_0$ and let $\mathbf{X} \subseteq \mathbf{F}_0$ be a
  neighborhood of $x$. Then a function $f \of \mathbf{X} \longrightarrow
  \mathbf{F}_1$ is said differentiable at $x$ if there is an $l \in
  \mathbf{F}_1$ such that
  \[ \forall \varepsilon \in \mathbf{F}_1^{\times}, \exists \delta \in
     \mathbf{F}_0^{\times}, \forall y \in B_0 (x, \delta), f (y) \in B_1 (f
     (x) - (y - x) l, (y - x) \varepsilon), \]
  i.e. $l$ is a limit at $0$ of the function $\mathbf{F}_0^{\times}
  \longrightarrow \mathbf{F}_1 ; h \mapsto \frac{f (x + h) - f (x)}{h}$.
\end{definition}

Then $l$ is unique, we write $l = f' (x)$ and we call $f' (x)$ the derivative
of $f$ at $x$. If moreover $\mathbf{X}$ is open and $f$ is differentiable at
each $x \in \mathbf{X}$, then we say that $f$ is
{\tmem{differentiable}}{\index{differentiable function}} and we write $f'$ for
the function $\mathbf{X} \longrightarrow \mathbf{F}_1 ; y \mapsto f' (y)$.

Many elementary properties of differentiable functions on $\mathbb{R}$ are
retained in the more general context of valued fields. In particular, the sum
and product of differentiable functions at a point is differentiable at this
point. Moreover, for $f, g$ differentiable at $x$ (resp. on $\mathbf{O}$), we
have
\[ (fg)' (x) = f' (x) g (x) + f (x) g' (x) . \]
In other words, the derivation operator $f \mapsto f' (x)$ behaves as a
derivation on the ring of differentiable functions at $x$. We also have a
chain rule: if $f \of \mathbf{O} \longrightarrow \mathbf{U} \subseteq
\mathbf{F}_1$ is differentiable at $x$ where $\mathbf{U}$ is a neighborhood of
$f (x)$, and $g \of \mathbf{U} \longrightarrow \mathbf{F}_2$ is differentiable
at $f (x)$ where $\mathbf{F}_2$ is a valued field, then $g \circ f$ is
differentiable at $x$ with
\begin{equation}
  (g \circ f)' (x) = g' (f (x)) f' (x) . \label{eq-chain-rule-functions}
\end{equation}
See {\cite{BGM}} for more details on these facts.

\subsection{Power series}

Let $\mathbf{F}$ be a field and let $\mathbf{A}$ be an algebra over
$\mathbf{F}$. Seeing $\mathbb{N}$ as the ordered monoid $(\mathbb{N}, +, 0,
>)$, we have an algebra $\mathbf{A} \llbracket X \rrbracket \assign \mathbf{A}
\llbracket X^{\mathbb{N}} \rrbracket$ of Noetherian series corresponding to
the algebra of formal power series over $\mathbf{A}$. It is equipped with a
standard derivation
\[ P = \sum_{k \geqslant 0} P_k X^k \longmapsto P' \assign \sum_{k \geqslant
   0} (k + 1) P_{k + 1} X^k . \]
Moreover, for $P, Q \in \mathbf{A} \llbracket X \rrbracket$ with $Q_0 = 0$ (in
other words, $Q \in X\tmmathbf{\Alpha} \llbracket X \rrbracket$), we have a
{\tmem{composite power series}}{\index{composite power series}}
\[ P \circ Q \assign P_0 + \sum_{k \in \mathbb{N}} \left( \sum_{m_1 + \cdots +
   m_n = k} P_n Q_{m_1} \cdots Q_{m_n} \right) X^k \in \mathbf{A} \llbracket X
   \rrbracket . \]
For $P \in \mathbf{A} \llbracket X \rrbracket$ and $Q, R \in X \mathbf{A}
\llbracket X \rrbracket$, we have $Q \circ R \in X \mathbf{A} \llbracket X
\rrbracket$ and
\[ P \circ (Q \circ R) = (P \circ Q) \circ R. \]

\subsection{Convergence of power series}

We fix a linearly ordered Abelian group $\mathfrak{M}$, a field $\mathbf{K}$,
and consider the field $\mathbb{S} \assign \mathbf{K} \llbracket \mathfrak{M}
\rrbracket$.

\begin{definition}
  Given a power series
  \[ P = \sum_{k \in \mathbb{N}} P_k X^k \in \mathbb{S} \llbracket X
     \rrbracket, \]
  and $s \in \mathbb{S}$, we say that $P$ {\tmstrong{{\tmem{converges at}}}}
  $s${\index{convergence of a power series}} if the family $(P_k s^k)_{k \in
  \mathbb{N}}$ is summable. We then set\label{autolab16}
  \[ \tilde{P} (s) \assign \sum_{k \in \mathbb{N}} P_k s^k . \]
  We write $\tmop{Conv} (P)$\label{autolab17} for the class of series $s \in
  \mathbb{S}$ at which $P$ converges. 
\end{definition}

\begin{example}
  Any power series $P = \sum_{k \in \mathbb{N}} c_k X^k \in \mathbf{K}
  \llbracket X \rrbracket$ converges on $\mathbb{S}^{\prec}$ by
  Lemma~\ref{lem-Neumann-series}. In fact, since the sequence $(s^k)_{k \in
  \mathbb{N}}$ is $\preccurlyeq$-increasing whenever $s \succcurlyeq 1$, we
  have $\tmop{Conv} (P) =\mathbb{S}^{\prec}$ unless $P$ is a polynomial, in
  which case $\tmop{Conv} (P) =\mathbb{S}$.
\end{example}

\begin{proposition}
  \label{prop-analytic-convex}{\tmem{{\cite[Corollary~1.5.8]{Schm01}}}} For
  all $P \in \mathbb{S} \llbracket X \rrbracket$, and $\varepsilon, \delta \in
  \mathbb{S}$ with $\delta \in \tmop{Conv} (P)$, we have $\varepsilon
  \preccurlyeq \delta \Longrightarrow \varepsilon \in \tmop{Conv} (P)$.
\end{proposition}

\begin{proof}
  Write $P = \sum_{k \in \mathbb{N}} P_k X^k$ and $u \assign \varepsilon /
  \delta \preccurlyeq 1$. By Proposition~\ref{prop-adding-finite-factors} for
  $I =\mathbb{N}$ and $f = \tmop{id}_{\mathbb{N}}$, the family $(P_k \delta^k
  u^k)_{k \in \mathbb{N}} = (P_k \varepsilon^k)_{k \in \mathbb{N}}$ is
  summable.
\end{proof}

In particular, $\tmop{Conv} (P)$ is a union of balls, hence the following.

\begin{corollary}
  \label{cor-open-conv}Let $P \in \mathbb{S} \{ \!\! \{X\} \!\! \}$. Then
  $\tmop{Conv} (P)$ is an open additive subgroup of $\mathbb{S} \left\{ \!\! 
  \!\!\! \{ X \} \!\! \right\}$.
\end{corollary}

We say that a $P \in \mathbb{S} \llbracket X \rrbracket$ is
{\tmem{convergent}} if $\tmop{Conv} (P) \neq \{0\}$, and we write $\mathbb{S}
\{ \!\! \{X\} \!\! \}$\label{autolab18} for the class of convergent power
series.

\begin{lemma}
  \label{lem-conv-sum-prod}Let $P, Q \in \mathbb{S} \{ \!\! \{X\} \!\! \}$.
  Then $\tmop{Conv} (P) \cap \tmop{Conv} (Q) \subseteq \tmop{Conv} (P + Q)$,
  with equality if $\tmop{Conv} (P) \neq \tmop{Conv} (Q)$. Moreover,
  $\tmop{Conv} (P) \cap \tmop{Conv} (Q) \subseteq \tmop{Conv} (PQ)$.
\end{lemma}

\begin{proof}
  For $\delta \in \tmop{Conv} (P) \cap \tmop{Conv} (Q)$, the families $(P (k)
  \delta^k)_{k \in \mathbb{N}}$ and $(Q (k) \delta^k)_{k \in \mathbb{N}}$ are
  summable, so by Proposition~\ref{lem-sum-sum} so is $((P_k + Q_k)
  \delta^k)_{k \in \mathbb{N}}$. So $\delta \in \tmop{Conv} (P + Q)$.
  
  The family $(P_k Q_n \delta^{k + n})_{k, n \in \mathbb{N}}$ is summable by
  Proposition~\ref{prop-sum-prod}. Therefore $(\sum_{k + n = m} \binom{m}{k} P
  (k) Q (n) \delta^m)_{m \in \mathbb{N}}$ is summable by
  Proposition~\ref{prop-summation-by-parts}. So $\delta \in \tmop{Conv}
  (PQ)$. If $\tmop{Conv} (P) \neq \tmop{Conv} (Q)$, then by
  Proposition~\ref{prop-analytic-convex}, we may assume that $\tmop{Conv} (P) \subseteq
  \tmop{Conv} (Q)$, so $\tmop{Conv} (P) \cap \tmop{Conv} (Q) = \tmop{Conv}
  (P)$. For $\delta \in \tmop{Conv} (Q) \setminus \tmop{Conv} (P)$, if $((P_k
  + Q_k) \delta^k)_{k \in \mathbb{N}}$, were summable, then so would be $(P_k
  \delta^k)_{k \in \mathbb{N}} = ((P_k + Q_k) \delta^k - Q_k \delta^k)_{k \in
  \mathbb{N}}$ by Proposition~\ref{lem-sum-sum}: a contradiction. So
  $\tmop{Conv} (P + Q) = \tmop{Conv} (P)$.
\end{proof}

\begin{corollary}
  The class $\mathbb{S} \{ \!\! \{X\} \!\! \}$ is a subalgebra of $\mathbb{S}
  \llbracket X \rrbracket$ containing $\mathbb{S} \cup \{X\}$.
\end{corollary}

In the sequel, we assume that $\mathbf{K}$ has characteristic zero.

\begin{lemma}
  \label{lem-derivative-power-series-radius}For all $P \in \mathbb{S}
  \llbracket X \rrbracket$ and $n \in \mathbb{N}$, we have $\tmop{Conv} (P) =
  \tmop{Conv} (P^{(n)})$.
\end{lemma}

\begin{proof}
  It suffices to prove the result for $n = 1$. We have $0 \in \tmop{Conv} (P)
  \cap \tmop{Conv} (P')$. Recall that $P' = \sum_{k \in \mathbb{N}} (k + 1)
  P_{k + 1} X^k$. For $\varepsilon \in \mathbb{S}^{\times}$, the family $(P_k
  \varepsilon^k)_{k \in \mathbb{N}}$ is summable if and only if $(P_{k + 1}
  \varepsilon^{k + 1})_{k \in \mathbb{N}}$ is summable. Since $\mathbf{K}$ has
  characteristic zero, this is equivalent to $((k + 1) P_{k + 1}
  \varepsilon^k)_{k \in \mathbb{N}}$ being summable. We deduce that
  $\tmop{Conv} (P) = \tmop{Conv} (P')$.
\end{proof}

\begin{corollary}
  The algebra $\mathbb{S} \{ \!\! \{X\} \!\! \}$ is a differential subalgebra
  of $(\mathbb{S} \llbracket X \rrbracket, \mathord{\:'})$.
\end{corollary}

\begin{proposition}
  \label{prop-power-series-translation}Let $P = \sum_{k \in \mathbb{N}} P_k
  X^k \in \mathbb{S} \llbracket X \rrbracket$ be a power series and let
  $\varepsilon, \delta \in \tmop{Conv} (P)$. Write $P_{+ \varepsilon}$ for the
  power series
  \[ P_{+ \varepsilon} \assign \sum_{k \in \mathbb{N}}
     \frac{\widetilde{P^{(k)}} (\varepsilon)}{k!} X^k . \]
  We have $\delta \in \tmop{Conv} (P_{+ \varepsilon})$ and
  \[ \widetilde{P_{+ \varepsilon}} (\delta) = \tilde{P} (\varepsilon + \delta)
     . \]
\end{proposition}

\begin{proof}
  Note that $P_{+ 0} = P$ and that $P_{+ \varepsilon} (0) = \tilde{P}
  (\varepsilon)$, so we may assume that $\varepsilon$ and $\delta$ are
  non-zero. The power series $P_{+ \varepsilon}$ is well-defined by
  Lemma~\ref{lem-derivative-power-series-radius}. We have trivially that
  \[ \bigcup_{i, k \in \mathbb{N}} \tmop{supp} (P_{k + i} \varepsilon^{k + i})
     = \bigcup_{j \in \mathbb{N}} \tmop{supp} (P_j \varepsilon^j), \]
  where the right hand set is well-based since $(P_j \varepsilon^j)_{j \in
  \mathbb{N}}$ is summable. For each monomial $\mathfrak{m} \in \mathfrak{M}$,
  the set $I_{\mathfrak{m}} \assign \{ (i, k) \in \mathbb{N}^2 \suchthat
  \mathfrak{m} \in \tmop{supp} (P_{k + i} \delta^{k + i}) \}$ is contained in
  $\{ (i, k) \in \mathbb{N}^2 \suchthat i + k \in J_{\mathfrak{m}} \}$ where
  \[ J_{\mathfrak{m}} \assign \{ j \in \mathbb{N} \suchthat \mathfrak{m} \in
     \tmop{supp} (P_j \delta^j) \} . \]
  Since $(P_j \varepsilon^j)_{j \in \mathbb{N}}$ is summable, we deduce that
  $J_{\mathfrak{m}}$, and hence $I_{\mathfrak{m}}$ are finite. This shows that
  $(P_{k + i} \varepsilon^{k + i})_{i, k \in \mathbb{N}}$ is summable.
  Likewise, $(P_{k + i} \delta^{k + i})_{i, k \in \mathbb{N}}$ is summable.
  
  For $k \in \mathbb{N}$, we have
  \begin{equation}
    \frac{\widetilde{P^{(k)}} (\varepsilon)}{k!} \delta^k = \sum_{i \in
    \mathbb{N}} \binom{k + i}{k} P_{k + i} \varepsilon^i \delta^k .
    \label{eq-binom}
  \end{equation}
  Therefore it suffices to show that the family $(P_{k + i} \varepsilon^i
  \delta^k)_{i, k \in \mathbb{N}}$ is summable in order to prove that~$\delta
  \in \tmop{Conv} (P_{+ \varepsilon})$. For $i, k \in \mathbb{N}$, write
  \[ \varepsilon^i \delta^k = u^{i + k} v^k \]
  where $(u, v) = \left( \varepsilon, \delta / \varepsilon \right)$ if $\delta
  \preccurlyeq \varepsilon$ and $(u, v) = \left( \delta, \varepsilon / \delta
  \right)$ if $\varepsilon \prec \delta$. In any case, we have $v \preccurlyeq
  1$ and the family $(P_{k + i} u^{k + i})_{i, k \in \mathbb{N}}$ is summable.
  Applying Proposition~\ref{prop-adding-finite-factors} for $I =\mathbb{N}
  \times \mathbb{N}$ and $f = (a, b) \mapsto a + b$, we see that the family
  $(P_{k + i} u^{k + i} v^k)_{i, k \in \mathbb{N}} = (P_{k + i} \varepsilon^i
  \delta^k)_{i, k \in \mathbb{N}}$ is summable.
  
  On the other hand we have $\delta + \varepsilon \preccurlyeq \varepsilon$ or
  $\delta + \varepsilon \preccurlyeq \delta$, so $\delta + \varepsilon \in
  \tmop{Conv} (P)$ and $(P_k  (\delta + \varepsilon)^k)_{k \in \mathbb{N}}$ is
  summable. By Lemma~\ref{lem-Fubini}, we have
  \begin{eqnarray*}
    \sum_{k \in \mathbb{N}} \frac{\widetilde{P^{(k)}} (\varepsilon)}{k!}
    \delta^k & = & \sum_{k \in \mathbb{N}} \sum_{i \in \mathbb{N}} \binom{k +
    i}{k} P_{k + i} \varepsilon^i \delta^k\\
    & = & \sum_{i, k \in \mathbb{N}} \binom{k + i}{k} P_{k + i} \varepsilon^i
    \delta^k\\
    & = & \sum_{j \in \mathbb{N}} \sum_{l \leqslant j} \binom{j}{l} P_j
    \varepsilon^{j - l} \delta^l\\
    & = & \sum_{j \in \mathbb{N}} P_j  (\varepsilon + \delta)^j\\
    & = & \tilde{P} (\varepsilon + \delta),
  \end{eqnarray*}
  as desired.
\end{proof}

\begin{lemma}
  \label{lem-true-derivative-power-series}Let $P \in \mathbb{S} \{ \!\! \{X\}
  \!\! \}$. The function $\tilde{P}$ is infinitely differentiable on
  $\tmop{Conv} (P)$ with $\tilde{P}^{(n)} = \widetilde{P^{(n)}}$ on
  $\tmop{Conv} (P)$ for all $n \in \mathbb{N}$.
\end{lemma}

\begin{proof}
  Recall by Corollary~\ref{cor-open-conv} that $\tmop{Conv} (P)$ is open. We
  first prove that $\tilde{P}$ is differentiable on $\tmop{Conv} (P)$ with
  $\tilde{P}' = \widetilde{P'}$. Let $s \in \tmop{Conv} (P)$ and $\varepsilon
  \in \mathbb{S}^{\times}$. For all $h \in \mathbb{S}$ with $h \preccurlyeq
  s$, Proposition~\ref{prop-power-series-translation} yields
  \begin{eqnarray*}
    \frac{\tilde{P} (s + h) - \tilde{P} (s)}{h} & = & \sum_{k > 0}
    \frac{\widetilde{P^{(k)}} (s)}{k!} h^{k - 1}\\
    & = & \widetilde{P'} (s) + hu,
  \end{eqnarray*}
  where $u \assign \sum_{k \in \mathbb{N}} \frac{\widetilde{P^{(k + 2)}}
  (s)}{(k + 2) !} h^k$. We have $u \preccurlyeq \widetilde{P^{(k + 2)}} (s)
  s^k \backassign v$ for a $k \in \mathbb{N}$. Seting $\delta \assign
  \varepsilon / v$, we obtain $\frac{\tilde{P} (s + h) - \tilde{P} (s)}{h} -
  \widetilde{P'} (s) = hu \preccurlyeq \varepsilon$ whenever $h \preccurlyeq
  \delta$. So $\tilde{P}$ is differentiable at $s$ with $\tilde{P}' (s) =
  \widetilde{P'} (s)$. The result for all $n$ follows by induction.
\end{proof}

\begin{proposition}
  \label{prop-Taylor-composition-pow}Let $\mathbf{U} \subseteq \mathbb{S}$ be
  open. Let $P = \sum_{k \in \mathbb{N}} P_k X^k \in \mathbb{S} \left\{ \!\!
  \{ X \} \!\! \right\}$ and $Q = \sum_{k > 0} Q_k X^k \in X\mathbb{S} \left\{
  \!\! \{ X \} \!\! \right\}$. Let $\varepsilon_P \in \tmop{Conv} (P)$ and
  $\varepsilon \in \tmop{Conv} (Q)$ with
  \begin{equation}
    \forall m > 0, Q_m \varepsilon^m \prec \varepsilon_P .
    \label{eq-analytic-comp-cond}
  \end{equation}
  Then $\varepsilon \in \tmop{Conv} (P \circ Q)$, and $(\widetilde{P \circ Q})
  (\varepsilon) = \tilde{P} (\tilde{Q} (\varepsilon))$.
\end{proposition}

\begin{proof}
  For $n \in \mathbb{N}$ and $k \in \mathbb{N}^{>}$, set $X_{n, k} \assign \{
  v \in (\mathbb{N}^{>})^n \suchthat | v | = k \}$ and
  \[ c_{n, k} \assign \sum_{v \in X_{n, k}} P_n Q_{v_{[1]}} \cdots
     Q_{v_{[n]}}, \]
  so $P \circ Q = P_0 + \sum_{k > 0} \left( \sum_{n \geqslant 0} c_{n, k}
  \right) X^k$. Note that since $\varepsilon \in \tmop{Conv} (Q)$, the set
  \[ \mathfrak{S}_Q \assign \bigcup_{m \in \mathbb{N}} \tmop{supp} (Q_m
     \varepsilon^m) \]
  is well-based. We have $\mathfrak{S}_Q \prec \varepsilon_P$ by
  (\ref{eq-analytic-comp-cond}). Let $\mathfrak{m} \assign
  \mathfrak{d}_{\varepsilon}$. The set $\mathfrak{S}_P \assign \bigcup_{n \in
  \mathbb{N}} \tmop{supp} (P_n \mathfrak{m}^n)$ is well-based. For $n \in
  \mathbb{N}$ and $k \in \mathbb{N}^{>}$, we have
  \[ \tmop{supp} c_{n, k} \varepsilon^k \subseteq (\mathfrak{S}_Q \cdot
     \mathfrak{m}^{- 1})^n \cdot \mathfrak{S}_P, \]
  where $\mathfrak{S}_Q \cdot \mathfrak{m}^{- 1}$ is well-based and
  infinitesimal, and $\mathfrak{S}_P$ is well-based. Since each family $(c_{n,
  k} \varepsilon^k)_{k > 0}$ for $n \in \mathbb{N}$ is summable with sum
  $\tilde{Q} (\varepsilon)^n$. Applying Lemma~\ref{lem-well-based-fin-family} for
  $f (n, k) = n$, we conclude that $(c_{n, k} \varepsilon^k)_{n \geqslant 0, k
  > 0}$ is summable. We deduce by Lemma~\ref{lem-Fubini} that
  \begin{eqnarray*}
    \tilde{P} (\tilde{Q} (\varepsilon)) & = & \sum_{n \geqslant 0} P_n 
    \tilde{Q} (\varepsilon)^n\\
    & = & \sum_{n \geqslant 0} P_n  \left( \sum_{k > 0} Q_k \varepsilon^k
    \right)^n\\
    & = & P_0 + \sum_{n \geqslant 0} \sum_{k > 0} c_{n, k} \varepsilon^k\\
    & = & P_0 + \sum_{k > 0} \left( \sum_{n \geqslant 0} c_{n, k} \right)
    \varepsilon^k\\
    & = & (\widetilde{P \circ Q}) (\varepsilon) .
  \end{eqnarray*}
  This concludes the proof.
\end{proof}

\begin{corollary}
  \label{cor-power-series-translation}Let $P \in \mathbb{S} \{ \!\! \{X\} \!\!
  \}$ and let $\delta, \varepsilon \in \tmop{Conv} (P)$. We have $\tmop{Conv}
  (P_{+ \delta}) = \tmop{Conv} (P)$ and $P_{+ (\delta + \varepsilon)} = (P_{+
  \delta})_{+ \varepsilon}$.
\end{corollary}

\begin{proof}
  We may assume that $\delta \neq 0$.
  Proposition~\ref{prop-power-series-translation} shows that $\tmop{Conv}
  (P_{+ \delta}) \supseteq \tmop{Conv} (P)$ and that $(P_{+ \delta})_{+
  \varepsilon}$ is well-defined. Since $\delta \in \tmop{Conv} (P_{+
  \delta})$, Propositions~\ref{prop-analytic-convex} and
  \ref{prop-power-series-translation} yield
  \[ \widetilde{(P_{+ \delta})_{+ \varepsilon}} (\iota) = \widetilde{P_{+
     \delta}} (\varepsilon + \iota) = \tilde{P} (\delta + \varepsilon + \iota)
  \]
  for all $\iota \in \tmop{Conv} (P)$. We deduce by
  Proposition~\ref{prop-lifting-inequalities} that $P_{+ (\delta + \varepsilon)} = (P_{+
  \delta})_{+ \varepsilon}$. Applying
  Proposition~\ref{prop-power-series-translation}, this time to $(P_{+
  \delta}, - \delta)$, we get $\tmop{Conv} (P_{+ \delta}) \subseteq
  \tmop{Conv} (P)$, hence the equality.
\end{proof}

\subsection{Zeroes of power series}

We next consider zeros of power series functions. A {\tmem{zero}}{\index{zero
of a power series}} of a power series $P \in \mathbb{S} \llbracket X
\rrbracket$ is an element $s \in \tmop{Conv} (P)$ with $\tilde{P} (s) = 0$. We
still assume that $\mathbf{K}$ has characteristic zero.

\begin{proposition}
  \label{prop-lifting-inequalities}Let $P \in \mathbb{S} \{ \!\! \{X\} \!\!
  \}$ and let $\delta \in \tmop{Conv} (P) \setminus \{0\}$. If $\tilde{P}
  (\varepsilon) = 0$ for all $\varepsilon \preccurlyeq \delta$ then $P = 0$.
\end{proposition}

\begin{proof}
  We have $(\tilde{P})^{(n)} (0) = 0$ for all $n \in \mathbb{N}$ since
  $\tilde{P}$ is constant around $0$. It follows by
  Lemma~\ref{lem-true-derivative-power-series} that $\widetilde{(P^{(n)})} (0) = 0$
  for all $n \in \mathbb{N}$, so $P = P_{+ 0} = 0$.\end{proof}

\subsection{Analytic functions}\label{subsection-analytic-functions}

Assume that $\mathbf{K}$ has characteristic zero. Let $\mathbb{S}= \mathbf{K}
\llbracket \mathfrak{M} \rrbracket$ be a fixed field of well-based series over
$\mathbf{K}$ where $\mathfrak{M}$ is non-trivial. We also fix a non-empty open
subclass $\mathbf{O}$ of $\mathbb{S}$.

\begin{definition}
  Let $f \of \mathbf{O} \longrightarrow \mathbb{S}$ be a function and let $s
  \in \mathbf{O}$. We say that $f$ is {\tmem{{\tmstrong{analytic
  at}}}}{\index{analyticity}} $s$ if there are a convergent power series $f_s
  \in \mathbb{S} \left\{ \!\! \{ X \} \!\! \right\}$ and a $\delta \in
  \tmop{Conv} (f_s) \setminus \{ 0 \}$ such that for all $\varepsilon
  \preccurlyeq \delta$, we have
  \[ s + \varepsilon \in \mathbf{O} \Longrightarrow f (s + \varepsilon) =
     \widetilde{f_s} (\varepsilon) . \]
  We say that $f_s$ is a {\tmem{{\tmstrong{Taylor series}}}}{\index{Taylor
  series}} of $f$ at $s$. We say that $f$ is {\tmem{{\tmstrong{analytic}}}} if
  it is analytic at each~$s \in \mathbf{O}$.
\end{definition}

\begin{example}
  The function $\tilde{P}$ induced by a convergent power series $P \in
  \mathbb{S} \left\{ \!\! \{ X \} \!\! \right\}$ is analytic on $\tmop{Conv}
  (P)$, by definition.
\end{example}

\begin{lemma}
  Let $f \of \mathbf{O} \longrightarrow \mathbb{S}$ be analytic at $s \in
  \mathbf{O}$. Then $f_s$ is the unique Taylor series of $f$ at $s$.
\end{lemma}

\begin{proof}
  Let $P \in \mathbb{S} \left\{ \!\! \{ X \} \!\! \right\}$ and $\delta \in
  \tmop{Conv} (P) \setminus \{ 0 \}$ with $s + \varepsilon \in \mathbf{O}$ and
  $f (s + \varepsilon) = \tilde{P} (\varepsilon)$ for all $\varepsilon
  \preccurlyeq \delta$. Then the function $\widetilde{f_s - P}$ is zero on the
  class of series $\varepsilon \preccurlyeq \delta$, so we have $f_s = P$ by
  Proposition~\ref{prop-lifting-inequalities}.
\end{proof}

If $f \of \mathbf{O} \longrightarrow \mathbb{S}$ is analytic at $s \in
\mathbf{O}$ where $\mathbf{O}$ is open, then we can define\label{autolab19}
\[ \tmop{Conv} (f)_s \assign \{ t \in \mathbf{O} \suchthat t - s \in
   \tmop{Conv} (f_s) \wedge f (t) = \widetilde{f_s} (t - s) \} . \]
\begin{proposition}
  \label{prop-analytic-power-series}Let $P \in \mathbb{S} \left\{ \!\! \{ X \}
  \!\! \right\}$. Then $\tilde{P}$ is analytic on $\tmop{Conv} (P)$ with
  $\tilde{P}_{\delta} = P_{+ \delta}$ and $\tmop{Conv} (\tilde{P})_{\delta} =
  \tmop{Conv} (P)$ for all $\delta \in \tmop{Conv} (P)$.
\end{proposition}

\begin{proof}
  Let $\delta \in \tmop{Conv} (P)$. The class $\tmop{Conv} (P)$ is open by
  Corollary~\ref{cor-open-conv}, with $\tmop{Conv} (P_{+ \delta}) =
  \tmop{Conv} (P)$. By Proposition~\ref{prop-power-series-translation}, we
  have $\tilde{P} (\delta + \varepsilon) = \widetilde{P_{+ \delta}
  (\varepsilon)}$ for all $\varepsilon \in \tmop{Conv} (P)$, so $\tilde{P}$ is
  indeed analytic on $\tmop{Conv} (P)$ with $\tmop{Conv} (\tilde{P})_{\delta}
  \supseteq \tmop{Conv} (P_{+ \delta}) = \tmop{Conv} (P)$. But we also have
  $\tmop{Conv} (\tilde{P})_{\delta} \subseteq \tmop{Conv} (P_{+ \delta}) =
  \tmop{Conv} (P)$ by definition, hence the result.
\end{proof}

\begin{proposition}
  \label{prop-analicity-strong}Let $f \of \mathbf{O} \longrightarrow
  \mathbb{S}$ be analytic at $s \in \mathbf{O}$ and let $\mathbf{U} \subseteq
  \tmop{Conv} (f)_s$ be a non-empty open subclass containing $0$. Then $f$ is
  analytic on $s + \mathbf{U}$, with $f_{s + \delta} = (f_s)_{+ \delta}$ for
  all $\delta \in \mathbf{U}$.
\end{proposition}

\begin{proof}
  Let $\delta \in \mathbf{U}$ and set $t \assign s + \delta$. Since
  $\mathbf{U} \ni 0$ is open and non-empty, we find a $\rho \neq 0$ with
  $\delta + \varepsilon \in \mathbf{U}$ for all $\varepsilon \preccurlyeq
  \rho$. Thus $f (t + \varepsilon) = \widetilde{f_s} (\delta + \varepsilon)$
  whenever $\varepsilon \preccurlyeq \rho$. But given such $\varepsilon$, we
  have $\widetilde{f_s} (\delta + \varepsilon) = \widetilde{(f_s)_{+ \delta}}
  (\varepsilon)$ by Proposition~\ref{prop-power-series-translation}, whence
  \[ f (t + \varepsilon) = \widetilde{f_s} (\delta + \varepsilon) =
     \widetilde{(f_s)_{+ \delta}} (\varepsilon) . \]
  So $f$ is analytic at $t$ with $f_t = (f_s)_{+ (t - s)}$.
\end{proof}

\begin{proposition}
  \label{prop-analytic-main}Let $f \of \mathbf{O} \longrightarrow \mathbb{S}$
  be analytic at $s \in \mathbf{O}$. Then $f$ is infinitely differentiable at
  $s$, and each $f^{(n)}$ for $n \in \mathbb{N}$ is analytic at $s$ with
  $\tmop{Conv} (f^{(n)})_s \supseteq \tmop{Conv} (f)_s$. Moreover, we have
  \[ f_s = \sum_{k \in \mathbb{N}} \tfrac{f^{(k)} (s)}{k!} X^k . \]
\end{proposition}

\begin{proof}
  Recall that $\widetilde{f_s}$ is infinitely differentiable on $\tmop{Conv}
  (f_s)$. It follows since $\tmop{Conv} (f)_s$ is a neighborhood of $s$ that
  $f$ is infinitely differentiable at~$s$. By
  Lemma~\ref{lem-true-derivative-power-series}, each derivative
  $\widetilde{f_s}^{(n)}$ for $n \in \mathbb{N}$ is a power series function on
  $\tmop{Conv} (f_s)$, and is thus analytic on $\tmop{Conv} (f_s)$ by
  Proposition~\ref{prop-analytic-power-series}. \ By
  Lemma~\ref{lem-true-derivative-power-series}, given $\delta \in \tmop{Conv}
  (f)_s - s$, we have $f^{(n)} (s + \delta) = \widetilde{f_s}^{(n)} (\delta) =
  \widetilde{(f_s)^{(n)}} (\delta)$. Therefore $f^{(n)}$ is analytic at $s$
  with $f^{(n)}_s = (f_s)^{(n)}$ and $\tmop{Conv} (f^{(n)})_s \supseteq
  \tmop{Conv} (f)_s$. Write $f_s = \sum_{k \in \mathbb{N}} s_k X^k$. We have
  $f^{(k)} (s) = \widetilde{(f_s)}^{(k)} (0) = \widetilde{(f_s)^{(k)}} (0) =
  k!s_k$. We deduce that $f_s = \sum_{k \in \mathbb{N}} \tfrac{f^{(k)}
  (s)}{k!} X^k$.
  
  \ 
\end{proof}

\begin{proposition}
  \label{prop-piecewise-analycity}Let $\mathbf{O} \subseteq \mathbb{S}$ be
  open and non-empty and assume that $\mathbf{O} = \bigsqcup_{i \in
  \mathbf{I}} \mathbf{O}_i$ where each $\mathbf{O}_i$ is open and non-empty.
  Let $(s_i)_{i \in \mathbf{I}}$ be a family where $s_i \in \mathbf{O}_i$ for
  all $i \in \mathbf{I}$. Let $(P_i)_{i \in \mathbf{I}}$ be a family of
  convergent power series in $\mathbb{S} \left\{ \!\! \{ X \} \!\! \right\}$
  with $(s_i + \tmop{Conv} (P_i)) \supseteq \mathbf{O}_i$. The function $f \of
  \mathbf{O} \longrightarrow \mathbb{S}$ such that for all $i \in \mathbf{I}$
  and $s \in \mathbf{O}_i$, we have $f (s) = P_i (s - s_i)$ is well-defined
  and analytic.
\end{proposition}

\begin{proof}
  Let $s \in \mathbf{O}$ and let $i \in \mathbf{I}$ with $s \in \mathbf{O}_i$.
  We have $s - s_i \in \mathbf{O}_i - s_i \subseteq \tmop{Conv} (P_i)$ so
  $\widetilde{P_i} (s - s_i)$ is defined. In particular $f$ is well-defined.
  The class $\mathbf{O}_i - s_i$ is a neighborhood of $0$, so there is a
  $\delta \in \tmop{Conv} (P_i) \setminus \{ 0 \}$ such that $s_i +
  \varepsilon \in \mathbf{O}_i$ whenever $\varepsilon \preccurlyeq \delta$.
  Given $\varepsilon \preccurlyeq \delta$, we have
  \[ f (s + \varepsilon) = P_i (s + \varepsilon - s_i) = (P_i)_{+ (s - s_i)}
     (\varepsilon) \]
  by Proposition~\ref{prop-power-series-translation}. Therefore $f$ is
  analytic at $s$ with $f_s = (P_i)_{+ (s - s_i)}$.
\end{proof}

We leave it to the reader to check that analyticity, at a point or on an open
class, is preserved by sums and products. The following result can be used to
show that the compositum of analytic functions is analytic. As a corollary of
Proposition~\ref{prop-Taylor-composition-pow}, we obtain:

\begin{corollary}
  \label{cor-Taylor-composition}Let $\mathbf{U} \subseteq \mathbb{S}$ be open.
  Let $f \of \mathbf{U} \longrightarrow \mathbb{S}, g \of \mathbf{O}
  \longrightarrow \mathbf{U}$ and let $s \in \mathbf{O}$ such that $g$ is
  analytic at $s$ and $f$ is analytic at $g (s)$. Write $g_s = \sum_{n \in
  \mathbb{N}} a_k X^n$. Let $\varepsilon_f \in \tmop{Conv} (f)_{g (s)} - g
  (s)$ and $\varepsilon \in \tmop{Conv} (g)_s - s$ with $\forall k > 0, a_k
  \varepsilon^k \prec \varepsilon_f$. Then function $f \circ g$ is analytic at
  $s$ with $s + \varepsilon \in \tmop{Conv} (f \circ g)_s$, and $(f \circ g)_s
  = f_{g (s)} \circ (g_s - g (s))$.
\end{corollary}

\begin{remark}
  \label{rem-analytic}A well-known type of analytic functions is that of
  restricted real-analytic functions of {\cite{DvdD:an,vdDMM}}. Given a
  non-empty interval $I$ of $\mathbb{R}$ and $f \of I \longrightarrow \mathbb{R}$
  is an analytic function, then $f$ extends into a function $\overline{f} \of
  I +\mathbb{S}^{\prec} \longrightarrow \mathbb{R}+\mathbb{S}^{\prec}$
  by
  \[ \forall r \in I, \forall \varepsilon \prec 1, \overline{f} (r +
     \varepsilon) \assign \sum_{k \in \mathbb{N}} \frac{f^{(k)} (r)}{k!}
     \varepsilon^k . \]
  We say that $\overline{f}$ is a restricted real-analytic function on
  $\mathbb{S}$. The function $\overline{f}$ is in fact analytic.
\end{remark}

\begin{remark}
  Our notion of analyticity is local, which makes it subject to pathologies
  (see Proposition~\ref{prop-piecewise-analycity}). A stronger version of analyticity
  would be to impose that a function $f$ is analytic at $s \in \mathbb{S}$ if
  there is a power series $f_s \in \mathbb{S} \left\{ \!\! \{ X \} \!\!
  \right\}$ such that $f (s + \varepsilon) = \widetilde{f_s} (\varepsilon)$
  for $\varepsilon$ ranging in the {\tmem{whole locus of convergence}}
  $\tmop{Conv} (f_s)$ of $f_s$.
\end{remark}

\section{Algebras of Noetherian series given by cuts}

In this section, we assume that $\mathbf{K}$ is a field and that
$\mathfrak{M}$ is a linearly ordered Abelian group, so $\mathbb{S} \assign
\mathbf{K} \llbracket \mathfrak{M} \rrbracket$ is a field. Our main tool for
proving the strong linearity of operators involved in Taylor expansions is the
construction in \Cref{subsection-subalgebras} of algebras of formal series
over $\mathbb{S}$ related to a convergence condition given by a final segment
$\mathfrak{S}$ of $(\mathfrak{M}, \prec)$. A typical example would be the
interval $\mathfrak{S}= \{ \mathfrak{m} \in \mathfrak{M} \of \mathfrak{m}
\succ \mathfrak{n} \}$ for some fixed $\mathfrak{n}$. For instance, we can
construct a subalgebra of $\mathbb{S} \llbracket X \rrbracket$ whose elements
converge for all $\delta \prec \mathfrak{S}$.

\subsection{Algebras of formal power series given by
cuts}\label{subsection-subalgebras}

Let $\mathfrak{S}$ be a final segment of $(\mathfrak{M}, \prec)$. We will
define a partial ordering $\prec_{\mathfrak{S}}$ on the direct product
\[ \mathfrak{M} \times X^{\mathbb{Z}} \assign \{\mathfrak{m}X^k \suchthat
   \mathfrak{m} \in \mathfrak{M} \wedge k \in \mathbb{Z}\} . \]
It will extend to the smallest ordering on this product such that $X
\prec_{\mathfrak{S}} \mathfrak{S}$ in $\mathbf{K} \llbracket \mathfrak{M}
\times X^{\mathbb{Z}} \rrbracket$. Consider the subclass
\begin{equation}
  (\mathfrak{M} \times X^{\mathbb{Z}})^{\prec, \mathfrak{S}} \assign
  (\mathfrak{M}^{\prec} \times \{X^0 \}) \sqcup \{\mathfrak{m}X^k \suchthat k
  > 0 \wedge \exists \mathfrak{u} \in \mathfrak{S}, \mathfrak{m} \preccurlyeq
  \mathfrak{u}^{- k} \} \label{eq-negative-cone} .
\end{equation}
So for $(\mathfrak{m}, k) \in \mathfrak{M} \times \mathbb{N}$, we have
$\mathfrak{m}X^k \prec_{\mathfrak{S}} 1 \Longleftrightarrow \mathfrak{m}
\nsucc \mathfrak{S}^{- k}$. Recall that a strictly positive cone on an
Abelian, torsion-free group $(\mathcal{G}, \cdot, 1)$ is a subset $P \subseteq
\mathcal{G} \setminus \{1\}$ which is closed under products and such that $P
\cap P^{- 1} = \varnothing$. Such a cone induces a partial ordering $<_P$ on
$\mathcal{G}$ given by $f <_P g \Longleftrightarrow gf^{- 1} \in P$.

\begin{lemma}
  The class $(\mathfrak{M} \times X^{\mathbb{Z}})^{\prec, \mathfrak{S}}$ is a
  strictly positive cone on $\mathfrak{M} \cdot X^{\mathbb{Z}}$.
\end{lemma}

\begin{proof}
  By definition, the class $(\mathfrak{M} \times X^{\mathbb{Z}})^{\prec,
  \mathfrak{S}}$ does not contain $1 = 1 X^0$. Let $\mathfrak{m}X^k,
  \mathfrak{n}X^{k'} \in (\mathfrak{M} \times X^{\mathbb{Z}})^{\prec,
  \mathfrak{S}}$. We may assume without loss of generality that $k \leqslant
  k'$. If $k = k' = 0$, then $\mathfrak{m}, \mathfrak{n} \prec 1$ so
  $\mathfrak{m}X^k \mathfrak{n}X^{k'} =\mathfrak{m}\mathfrak{n} \prec 1$. If
  $k = 0$ and $k' \neq 0$, then $\mathfrak{m} \prec 1$, $k' > 0$ and there is
  a $\mathfrak{u} \in \mathfrak{S}$ with $\mathfrak{n} \preccurlyeq
  \mathfrak{u}^{- k'}$. We then have $\mathfrak{m}X^k \mathfrak{n}X^{k'} =
  (\mathfrak{m}\mathfrak{n}) X^{k'}$ where $\mathfrak{m}\mathfrak{n} \prec
  \mathfrak{n} \preccurlyeq \mathfrak{u}^{- k'}$. We deduce that
  $\mathfrak{m}X^k \mathfrak{n}X^{k'} \in (\mathfrak{M} \times
  X^{\mathbb{Z}})^{\prec, \mathfrak{S}}$. Otherwise, we must have $k, k' > 0$,
  and there are $(\mathfrak{v}, \mathfrak{w}) \in \mathfrak{S}$ such that
  $\mathfrak{m} \preccurlyeq \mathfrak{v}^{- k}$ and $\mathfrak{n}
  \preccurlyeq \mathfrak{w}^{- k'}$. Taking $\mathfrak{p} \assign \max
  (\mathfrak{v}, \mathfrak{w}) \in \mathfrak{S}$, we have
  $\mathfrak{m}\mathfrak{n} \preccurlyeq \mathfrak{p}^{- (k + k')}$, so
  $\mathfrak{m}X^k \mathfrak{n}X^{k'} \in (\mathfrak{M} \times
  X^{\mathbb{Z}})^{\prec, \mathfrak{S}}$. Thus $\mathfrak{m}X^k
  \mathfrak{n}X^{k'} \in (\mathfrak{M} \times X^{\mathbb{Z}})^{\prec,
  \mathfrak{S}}$ is closed under products.
  
  It remains to show that we cannot have $\mathfrak{m}^{- 1} X^{- k} \in
  (\mathfrak{M} \times X^{\mathbb{Z}})^{\prec, \mathfrak{S}}$. If $k = 0$,
  then this follows from the fact that $\mathfrak{M}^{\prec}$ is a strictly
  positive cone on $\mathfrak{M}$. Otherwise, we have $k > 0$ so
  $\mathfrak{m}^{- 1} X^{- k} \nin (\mathfrak{M} \times
  X^{\mathbb{Z}})^{\prec, \mathfrak{S}}$.
\end{proof}

We thus obtain a partial ordering $\prec_{\mathfrak{S}}$\label{autolab20} on
$\mathfrak{M} \cdot X^{\mathbb{N}} \subseteq \mathfrak{M} \cdot
X^{\mathbb{Z}}$ by setting
\[ \mathfrak{m}X^k \prec_{\mathfrak{S}} \mathfrak{n}X^{k'} \Longleftrightarrow
   \mathfrak{m}\mathfrak{n}^{- 1} X^{k - k'} \in (\mathfrak{M} \times
   X^{\mathbb{Z}})^{\prec, \mathfrak{S}} . \]
Mind that this is the {\tmem{reverse}} ordering of the ordering $<_P$ given by
the positive cone $P = (\mathfrak{M} \times X^{\mathbb{Z}})^{\prec,
\mathfrak{S}}$. We write $\mathfrak{M} \times_{\mathfrak{S}} X^{\mathbb{N}}$
for the corresponding partially ordered monoid. \ We may consider the algebra
of Noetherian series\label{autolab21}
\[ \mathbb{S} \llbracket X \rrbracket_{\mathfrak{S}} \assign \mathbf{K}
   \llbracket \mathfrak{M} \times_{\mathfrak{S}} X^{\mathbb{N}} \rrbracket \]
for this ordering.

\begin{lemma}
  We have a natural inclusion $\mathbb{S} \llbracket X
  \rrbracket_{\mathfrak{S}} \longrightarrow \mathbb{S} \llbracket X
  \rrbracket$ given by
  \[ s \mapsto \sum_{k \in \mathbb{Z}} \left( \sum_{\mathfrak{m}X^k \in
     \tmop{supp} s} s (\mathfrak{m}X^k)\mathfrak{m} \right) X^k . \]
\end{lemma}

\begin{proof}
  The identity is an embedding of $(\mathfrak{M} \times X^{\mathbb{N}},
  \prec_{\mathfrak{S}})$ into the lexicographic power $(\mathfrak{M} \times
  X^{\mathbb{N}}, \prec_{\tmop{lex}})$ with prevalence on $X^{\mathbb{N}}$, so
  Corollary~\ref{cor-Noetherian-subgroup} yields the inclusion.
\end{proof}

Under this inclusion, we have $\mathbb{S} \llbracket X
\rrbracket_{\mathfrak{S}} =\mathbb{S} \llbracket X \rrbracket$ if and only if
$\mathfrak{S}=\mathfrak{M}$ and $\mathbb{S} \llbracket X
\rrbracket_{\mathfrak{S}} =\mathbb{S} [X]$ if and only if~$\mathfrak{S}=
\varnothing$. In the divisible case, this generalises as follows:

\begin{lemma}
  \label{lem-increasing-cuts}Given a final segment $\mathfrak{T}$ of
  $\mathfrak{M}$, we have
  \[ \mathfrak{S} \subsetneq \mathfrak{T} \Longleftrightarrow \mathbb{S}
     \llbracket X \rrbracket_{\mathfrak{S}} \subsetneq \mathbb{S} \llbracket X
     \rrbracket_{\mathfrak{T}} . \]
\end{lemma}

\begin{proof}
  Assume that $\mathfrak{S} \subsetneq \mathfrak{T}$. Then the identity $(\mathfrak{M} \times X^{\mathbb{N}}, \prec_{\mathfrak{S}})
  \longrightarrow (\mathfrak{M} \times X^{\mathbb{N}}, \prec_{\mathfrak{T}})$ is an
  embedding,
  whence $\mathbb{S} \llbracket X \rrbracket_{\mathfrak{S}} \subseteq
  \mathbb{S} \llbracket X \rrbracket_{\mathfrak{T}}$ by
  Corollary~\ref{cor-Noetherian-subgroup}. Now let $\mathfrak{u} \in \mathfrak{T}
  \setminus \mathfrak{S}$, so $\mathfrak{u} \prec \mathfrak{S}$. We claim that
  the power series $P \assign \sum_{k \in \mathbb{N}} \mathfrak{u}^{- k} X^k$
  lies in $\mathbb{S} \llbracket X \rrbracket_{\mathfrak{T}} \setminus
  \mathbb{S} \llbracket X \rrbracket_{\mathfrak{S}}$. Indeed, we have
  \[ \mathfrak{u}^{- k}  (\mathfrak{u}^{- (k + 1)})^{- 1} = u \in
     \mathfrak{T}, \text{\qquad whereas\qquad$\mathfrak{u}^{- k} 
     (\mathfrak{u}^{- (k + 1)})^{- 1} =\mathfrak{u} \prec \mathfrak{S}$} . \]
  Thus $\mathfrak{u}^{- k} X^k \succ_{\mathfrak{T}} \mathfrak{u}^{- (k + 1)}
  X^{k + 1}$, but the same terms are not comparable for
  $\prec_{\mathfrak{S}}$. Hence the support of $P$ is Noetherian for
  $\succ_{\mathfrak{T}}$ but not for $\succ_{\mathfrak{S}}$, i.e. $P \in
  \mathbb{S} \llbracket X \rrbracket_{\mathfrak{T}} \setminus \mathbb{S}
  \llbracket X \rrbracket_{\mathfrak{S}}$. Recall that inclusion is a linear
  ordering on the collection of final segments of~$\mathfrak{M}$, so this
  concludes the proof.
\end{proof}

The main feature of $\mathbb{S} \llbracket X \rrbracket_{\mathfrak{S}}$ is
that its elements can be evaluated at series $\delta$ in extensions of
$\mathbb{S}$ such that $\delta \prec \mathfrak{S}$.

\begin{proposition}
  \label{prop-cut-evaluation}Let $\mathfrak{N} \supseteq \mathfrak{M}$ be an
  Abelian, linearly ordered group extension and write $\mathbb{T} \assign
  \mathbf{K} \llbracket \mathfrak{N} \rrbracket$, so we have a natural
  inclusion $\mathbb{S} \subseteq \mathbb{T}$. Let $\mathfrak{I}$ be a
  Noetherian subset of $(\mathfrak{M} \times X^{\mathbb{N}},
  \succ_{\mathfrak{S}})$ and let $\delta \in \mathbb{T}$ with $\delta \prec
  \mathfrak{S}$. Then the family $(\mathfrak{m} \delta^k)_{\mathfrak{m}X^k \in
  \mathfrak{I}}$ is summable in $\mathbb{T}$.
\end{proposition}

\begin{proof}
  Write $\mathfrak{v}=\mathfrak{d}_{\delta}$. Let $(\mathfrak{m}_i X^{k_i})_{i
  \in \mathbb{N}}$ be an injective sequence in $\mathfrak{I}$. Since
  $\mathfrak{I}$ is Noetherian, there are $i, j \in \mathbb{N}$ with $i < j$
  and $\mathfrak{m}_j X^{k_j} \prec_{\mathfrak{S}} \mathfrak{m}_i X^{k_i}$. If
  $k_i = k_j$, then this means that $\mathfrak{m}_j \prec \mathfrak{m}_i$, so
  $\mathfrak{m}_j \mathfrak{v}^{k_j} \prec \mathfrak{m}_i \mathfrak{v}^{k_i}$.
  Otherwise, we must have $k_i < k_j$, and $\mathfrak{m}_j \mathfrak{m}_i^{-
  1} \preccurlyeq \mathfrak{u}^{k_i - k_j}$ for a $\mathfrak{u} \in
  \mathfrak{S}$. Since $\mathfrak{v} \prec \mathfrak{S}$, we have
  $\mathfrak{m}_j \mathfrak{m}_i^{- 1} \prec \mathfrak{v}^{k_i - k_j}$, so
  $\mathfrak{m}_j \mathfrak{v}^{k_j} \prec \mathfrak{m}_i \mathfrak{v}^{k_i}$.
  We conclude with Lemma~\ref{lem-summable-bad-sequence} that
  $(\mathfrak{m}\mathfrak{v}^k)_{\mathfrak{m}X^k \in \mathfrak{I}}$ is
  summable. Since $\mathfrak{I}$ is Noetherian, the set $\{k \in \mathbb{N}
  \suchthat \exists \mathfrak{m} \in \mathfrak{M}, \mathfrak{m}X^k \in
  \mathfrak{I}\}$ must be well-ordered in $(\mathbb{Z}, <)$. It follows by
  Proposition~\ref{prop-adding-finite-factors} that $(\mathfrak{m}
  \delta^k)_{\mathfrak{m}X^k \in \mathfrak{I}}$ is summable.
\end{proof}

\begin{proposition}
  \label{prop-Noeth-evaluation}In the same notations as above, for all $\delta
  \in \mathbb{T}^{\times}$ with $\delta \prec \mathfrak{S}$, the function
  $\mathfrak{M} \times X^{\mathbb{N}} \longrightarrow \mathbb{T} \: ;
  \mathfrak{m}X^k \mapsto \mathfrak{m} \delta^k$ extends uniquely into a
  strongly linear morphism of algebras $\tmop{ev}_{\delta} \of \mathbb{S}
  \llbracket X \rrbracket_{\mathfrak{S}} \longrightarrow \mathbb{T}$.
\end{proposition}

\begin{proof}
  The function preserves products, so the result follows from
  Proposition~\ref{prop-strongly-linear-extension}.
\end{proof}

\begin{proposition}
  \label{prop-conv-duality}Assume that $\mathfrak{M}$ is divisible. Let
  $\mathfrak{N} \supseteq \mathfrak{M}$ be an Abelian linearly ordered group
  extension and set $\mathbb{T} \assign \mathbf{K} \llbracket \mathfrak{N}
  \rrbracket$. For $P = \sum_{k \in \mathbb{N}} \left(
  \sum_{\mathfrak{m} \in \mathfrak{M}} P_{k, \mathfrak{m}} \mathfrak{m}
  \right) X^k$ in $\mathbb{S} \llbracket X \rrbracket$ and $\delta \in
  \mathbb{T}^{\times}$, we have
  \[ P \in \mathbb{S} \llbracket X \rrbracket_{\{\mathfrak{m} \in \mathfrak{M}
     \suchthat \mathfrak{m} \succ \delta\}} \Longleftrightarrow (P_{k,
     \mathfrak{m}} \mathfrak{m} \delta^k)_{\mathfrak{m}X^k \in \mathfrak{M}
     \cdot X^{\mathbb{Z}}} \text{ is summable in $\mathbb{T}$.} \]
\end{proposition}

\begin{proof}
  We write $P_k = \sum_{\mathfrak{m} \in \mathfrak{M}} P_{k, \mathfrak{m}}$
  for each $k \in \mathbb{N}$. If $P \in \mathbb{S} \llbracket X
  \rrbracket_{\{\mathfrak{m} \in \mathfrak{M} \suchthat \mathfrak{m} \succ
  \delta\}}$, then $(P_{k, \mathfrak{m}} \mathfrak{m}
  \delta^k)_{\mathfrak{m}X^k \in \mathfrak{M} \cdot X^{\mathbb{Z}}}$ is
  summable by Proposition~\ref{prop-Noeth-evaluation}. Assume conversely that $(P_{k,
  \mathfrak{m}} \mathfrak{m} \delta^k)_{\mathfrak{m}X^k \in \mathfrak{M} \cdot
  X^{\mathbb{Z}}}$ is summable. Write $\mathfrak{d} \assign
  \mathfrak{d}_{\delta}$ and $\mathfrak{T} \assign \{\mathfrak{m} \in
  \mathfrak{M} \suchthat \mathfrak{m} \succ \mathfrak{d}\}$. Assume for
  contradiction that the support of $P$ is not Noetherian in $(\mathfrak{M}
  \times_{\mathfrak{S}} X^{\mathbb{Z}}, \succ_{\mathfrak{T}})$. So there is a
  bad sequence $(\mathfrak{m}_i X^{k_i})_{i \in \mathbb{N}}$ with
  $\mathfrak{m}_i \in \tmop{supp} P_{k_i}$ for all $i \in \mathbb{N}$. If
  $(k_i)_{i \in \mathbb{N}}$ were constant, then the sequence
  $(\mathfrak{m}_i)_{i \in \mathbb{N}}$ would witness that $\tmop{supp} P_k$
  is not Noetherian in $\mathfrak{M}$. So we may assume that $(k_i)_{i \in
  \mathbb{N}}$ is strictly increasing. For all $i, j \in \mathbb{N}$ with $i <
  j$, we have $\mathfrak{m}_i X^{k_i} \nsucccurlyeq \mathfrak{m}_j X^{k_j}$.
  This implies that $\mathfrak{m}_j \mathfrak{m}_i^{- 1} \succ
  \mathfrak{T}^{k_i - k_j}$, whence $\mathfrak{m}_j \mathfrak{m}_i^{- 1} \in
  (\mathfrak{M} \setminus \mathfrak{T})^{k_i - k_j}$ by divisibility of
  $\mathfrak{M}$. Thus $\mathfrak{m}_j \mathfrak{m}_i^{- 1} \succcurlyeq
  \mathfrak{d}^{k_i - k_j}$. Now the family $(\mathfrak{m}_i \delta^{k_i})_{i
  \in \mathbb{N}}$ is summable. Therefore there are $i < j$ with
  $\mathfrak{m}_i \delta^{k_i} \succ \mathfrak{m}_j \delta^{k_j}$, whence
  $\mathfrak{m}_i \mathfrak{d}^{k_i} \succ \mathfrak{m}_j \mathfrak{d}^{k_j}$:
  a contradiction. \end{proof}

\begin{remark}
  We do not have $P \in \mathbb{S} \llbracket X \rrbracket_{\mathfrak{M}
  \setminus \tmop{Conv} (P)}$ in general. For instance if $\mathfrak{M}=
  x^{\mathbb{Q}}$ is a multiplicative copy of $(\mathbb{Q}, +, 0, <)$, then
  the series $P = \sum_{k \in \mathbb{N}} x^{- 2^k} X^k$ satisfies
  $\tmop{Conv} (P) =\mathbb{S}$ but $P \nin \mathbb{S} \llbracket X
  \rrbracket_{\varnothing} =\mathbb{S} [X]$.
\end{remark}

\subsection{Cut extensions of algebra
morphisms}\label{subsection-cut-extensions}

Fix a non-trivial, {\tmem{linearly}} ordered {\tmem{Abelian group}}
$\mathfrak{N}$ and write $\mathbb{T} \assign \mathbf{K} \llbracket
\mathfrak{N} \rrbracket$. Let $\triangle \of \mathbb{S} \longrightarrow
\mathbb{T}$ be a strongly linear morphism of algebras. Let $\mathfrak{T}
\subseteq \mathfrak{N}$ be a non-empty final segment, and write
\[ \triangle^{\ast} (\mathfrak{T}) \assign \{\mathfrak{m} \in \mathfrak{M}
   \suchthat \exists \mathfrak{n} \in \mathfrak{T}, \mathfrak{m} \succcurlyeq
   \mathfrak{d}_{\triangle (\mathfrak{n})} \} . \]
Then $\triangle^{\ast} (\mathfrak{T})$ is a final segment of $\mathfrak{M}$,
so we have orderings $\prec_{\mathfrak{T}}$ and $\prec_{\triangle^{\ast}
(\mathfrak{T})}$ on $\mathfrak{N} \times X^{\mathbb{N}}$ and $\mathfrak{M}
\times X^{\mathbb{N}}$ respectively, and two corresponding algebras of
Noetherian series $\mathbb{S} \llbracket X \rrbracket_{\triangle^{\ast}
(\mathfrak{T})}$ and $\mathbb{T} \llbracket X \rrbracket_{\mathfrak{T}}$.

Note that $\prec_{\mathfrak{T}}$ and $\prec_{\triangle^{\ast} (\mathfrak{T})}$
extend the orderings on $\mathfrak{M}$ and $\mathfrak{N}$ respectively, so
$\triangle$ is an embedding $\mathbb{S} \longrightarrow \mathbb{T} \llbracket
X \rrbracket_{\mathfrak{T}}$. Consider the function
\begin{eqnarray*}
  \overline{\triangle} \of \mathfrak{M} \times_{\triangle^{\ast}
  (\mathfrak{T})} X^{\mathbb{N}} & \longrightarrow & \mathbb{T} \llbracket X
  \rrbracket_{\mathfrak{T}}\\
  \mathfrak{m}X^k & \longmapsto & \triangle (\mathfrak{m}) X^k
\end{eqnarray*}
\begin{proposition}
  \label{prop-Noeth-ext-comp}The function $\overline{\triangle}$ is
  Noetherian.
\end{proposition}

\begin{proof}
  Let $\mathfrak{I}$ be a Noetherian subset of $(\mathfrak{M} \times
  X^{\mathbb{N}}, \succ_{\triangle^{\ast} (\mathfrak{T})})$. We want to prove
  that the family $(\overline{\triangle} (\mathfrak{m}X^k))_{\mathfrak{m}X^k
  \in \mathfrak{I}}$ is summable. Let $(\mathfrak{m}_i X^{k_i})_{i \in
  \mathbb{N}}$ be an injective sequence in $\mathfrak{I}$ and let
  $(\mathfrak{n}_i)_{i \in \mathbb{N}} \in \mathfrak{N}^{\mathbb{N}}$ be a
  sequence with $\mathfrak{n}_i \in \tmop{supp} \triangle (\mathfrak{m}_i)$
  for all $i \in \mathbb{N}$. By Lemma~\ref{lem-summable-bad-sequence}, it suffices
  to show that there are $i, j \in \mathbb{N}$ with $i < j$ and
  $\mathfrak{n}_j X^{k_j} \prec_{\mathfrak{T}} \mathfrak{n}_i X^{k_i}$ . This
  condition is preserved under taking subsequences, so we may assume that
  $(\mathfrak{m}_i X^{k_i})_{i \in \mathbb{N}}$ is strictly decreasing for the
  ordering $\prec_{\triangle^{\ast} (\mathfrak{T})}$. For each $i \in
  \mathbb{N}$, the relation
  \begin{equation}
    \frac{\mathfrak{m}_{i + 1}}{\mathfrak{m}_i} X^{k_{i + 1} - k_i} \in
    (\mathfrak{M} \times X^{\mathbb{Z}})^{\prec, \triangle^{\ast}
    (\mathfrak{T})}, \label{eq-in-neg-cone}
  \end{equation}
  implies in particular that $k_{i + 1} \geqslant k_i$. Taking a subsequence
  if necessary, we may assume that $(k_i)_{i \in \mathbb{N}}$ is either
  constant or strictly increasing.
  
  In the constant case, the condition (\ref{eq-in-neg-cone}) reduces to
  $\mathfrak{m}_i \succ \mathfrak{m}_{i + 1}$, i.e. $(\mathfrak{m}_i)_{i \in
  \mathbb{N}}$ is strictly decreasing. But then since $\triangle$ is strongly
  linear, the family $(\triangle (\mathfrak{m}_i))_{i \in \mathbb{N}}$ is
  summable. By Lemma~\ref{lem-summable-bad-sequence}, there are $i \in \mathbb{N}$
  and $l > 0$ with $\mathfrak{n}_{i + l} \prec \mathfrak{n}_i$, whence
  $\mathfrak{n}_{i + l} X^{k_{i + l}} =\mathfrak{n}_{i + l} X^{k_0}
  \prec_{\mathfrak{T}} \mathfrak{n}_i X^{k_0} =\mathfrak{n}_i X^{k_i}$.
  
  In the strictly increasing case, the condition (\ref{eq-in-neg-cone})
  translates as $\frac{\mathfrak{m}_{i + 1}}{\mathfrak{m}_i} \preccurlyeq
  \mathfrak{u}_i^{- (k_{i + 1} - k_i)}$ for some $\mathfrak{u}_i \in
  \triangle^{\ast} (\mathfrak{T})$. Rewriting this as
  \[ \frac{\mathfrak{m}_i}{\mathfrak{u}_i^{- k_i}} \succcurlyeq
     \frac{\mathfrak{m}_{i + 1}}{\mathfrak{u}_i^{- k_{i + 1}}}, \]
  we have the following weakly decreasing sequence in $\mathfrak{M}$:
  \[ \frac{\mathfrak{m}_0}{\mathfrak{u}_0^{- k_0}} \succcurlyeq
     \frac{\mathfrak{m}_1}{\mathfrak{u}_0^{- k_1}} \succcurlyeq
     \frac{\mathfrak{m}_2}{\mathfrak{u}_0^{- k_1} \mathfrak{u}_1^{- (k_2 -
     k_1)}} \succcurlyeq \cdots \succcurlyeq \frac{\mathfrak{m}_{i +
     1}}{\mathfrak{u}_0^{- k_1} \mathfrak{u}_1^{- (k_2 - k_1)} \cdots
     \mathfrak{u}_i^{- (k_{i + 1} - k_i)}} \succcurlyeq \cdots . \]
  Write $\mathfrak{p}_i \assign \mathfrak{u}_0^{- k_1} \mathfrak{u}_1^{- (k_2
  - k_1)} \cdots \mathfrak{u}_i^{- (k_{i + 1} - k_i)}$ for each $i > 0$. Since
  $\triangle \upharpoonleft \mathfrak{M}$ is Noetherian,
  Lemma~\ref{lem-Noetherian-function} and Lemma~\ref{lem-summable-bad-sequence} for the
  sequence $\mathfrak{n}_i  (\mathfrak{d}_{\smash{\triangle
  (\mathfrak{p}_i)}})^{- 1} \in \tmop{supp} \triangle
  (\frac{\mathfrak{m}_i}{\mathfrak{p}_i})$ gives $i, j > 0$ with $i < j$ and
  \[ \frac{\mathfrak{n}_i}{\triangle (\mathfrak{p}_i)} \succcurlyeq
     \frac{\mathfrak{n}_j}{\triangle (\mathfrak{p}_j)}, \]
  whence
  \[ \mathfrak{n}_i \succcurlyeq \frac{\mathfrak{n}_j}{\mathfrak{d}_{\triangle
     (\mathfrak{u}_i)}^{- (k_{i + 1} - k_i)} \cdots \mathfrak{d}_{\triangle
     (\mathfrak{u}_{j - 1})}^{- (k_j - k_{j - 1})}} . \]
  Taking $\mathfrak{u} \assign \min (\mathfrak{u}_i, \ldots, \mathfrak{u}_{j -
  1})$, we obtain
  \[ \mathfrak{n}_i \succcurlyeq \frac{\mathfrak{n}_j}{\mathfrak{d}_{\triangle
     (\mathfrak{u})}^{- (k_j - k_i)}}, \]
  whence $\frac{\mathfrak{n}_j}{\mathfrak{n}_i} \preccurlyeq
  \mathfrak{d}_{\triangle (\mathfrak{u})}^{- (k_j - k_i)}$. But
  $\mathfrak{d}_{\triangle (\mathfrak{u})} \in \mathfrak{T}$, so this means
  that $\mathfrak{n}_j X^{k_j} \prec_{\mathfrak{T}} \mathfrak{n}_i X^{k_i}$.
  This concludes the proof.
\end{proof}

\begin{corollary}
  \label{cor-cut-extension-morphism}The function $\overline{\triangle}$
  extends into a strongly linear morphism of algebras $\overline{\triangle}
  \of \mathbb{S} \llbracket X \rrbracket_{\triangle^{\ast} (\mathfrak{T})}
  \longrightarrow \mathbb{T} \llbracket X \rrbracket_{\mathfrak{T}}$.
\end{corollary}

\section{Differential algebra}\label{section-differential-algebra}

We fix a field $\mathbf{K}$, and we recall that our algebras $(\mathbf{A}, +,
\cdot, 0, .)$ over $\mathbf{K}$ are always associative, but not necessarily
commutative or unital.

\subsection{Differential algebra}

We first recall standard and basic notions in differential algebra. The
results here are folklore and we give proofs for the sake of completion. Let
$\mathbf{B}$ be an algebra over $\mathbf{K}$ and let $\mathbf{A} \subseteq
\mathbf{B}$ be a subalgebra. A function $\partial \of \mathbf{A}
\longrightarrow \mathbf{B}$ is called a
{\tmem{derivation}}{\index{derivation}} if it is $\mathbf{K}$-linear and
satisfies the Leibniz product rule
\[ \forall a,b \in \mathbf{A},\partial (a \cdot b) = \partial (a) \cdot b + a \cdot \partial (b) .\]

\begin{example}
  If $\mathbf{A}$ is a $\mathbf{K}$-algebra, $a \in \mathbf{A}$, $\delta,
  \partial \of \mathbf{A} \longrightarrow \mathbf{A}$ are derivations and
  $\sigma \of \mathbf{A} \longrightarrow \mathbf{A}$ is an automorphism of
  algebra, then the following functions are derivations $\mathbf{A}
  \longrightarrow \mathbf{A}$:
  \begin{itemizedot}
    \item $\partial + a \cdot \delta \assign \: b \mapsto \partial (b) + a
    \cdot \delta (b)$,
    
    \item $[\partial, \delta] \assign \partial \circ \delta - \delta \circ
    \partial$,
    
    \item $[a, \cdot] \assign b \mapsto a \cdot b - b \cdot a$,
    
    \item $\sigma \circ \partial \circ \sigma^{\tmop{inv}} \assign b \mapsto
    \sigma (\partial (\sigma^{\tmop{inv}} (b)))$.
  \end{itemizedot}
\end{example}

\begin{lemma}[{\cite[Corollary~3.9]{vdH:noeth}}]\label{lem-strongly-lin-der}
  Suppose that $\mathbb{A}= \mathbf{K} \llbracket \mathfrak{M} \rrbracket$ and
  $\mathbb{B}= \mathbf{K} \llbracket \mathfrak{N} \rrbracket$ are algebras
  over $\mathbf{K}$ of Noetherian series and that $\partial \of \mathbb{A}
  \longrightarrow \mathbb{B}$ is a strongly linear function with
  \[ \partial (\mathfrak{m} \cdot \mathfrak{n}) = \partial (\mathfrak{m})
     \cdot \mathfrak{n}+\mathfrak{m} \cdot \partial (\mathfrak{n}) \]
  for all $\mathfrak{m}, \mathfrak{n} \in \mathfrak{M}$. Then $\partial$ is a
  derivation.
\end{lemma}

\begin{proof}
  Let $a, b \in \mathbb{A}$. We have
  \begin{eqnarray*}
    \partial (a \cdot b) & = & \partial \left( \sum_{\mathfrak{m},
    \mathfrak{n} \in \mathfrak{M}} a (\mathfrak{m}) b
    (\mathfrak{n})\mathfrak{m} \cdot \mathfrak{n} \right)\\
    & = & \sum_{\mathfrak{m}, \mathfrak{n} \in \mathfrak{M}} a (\mathfrak{m})
    b (\mathfrak{n}) \partial (\mathfrak{m} \cdot \mathfrak{n})
    \text{{\hspace*{\fill}}(by strong linearity)}\\
    & = & \sum_{\mathfrak{m}, \mathfrak{n} \in \mathfrak{M}} a (\mathfrak{m})
    b (\mathfrak{n})  (\partial (\mathfrak{m}) \cdot \mathfrak{n}+\mathfrak{m}
    \cdot \partial (\mathfrak{n}))\\
    & = & \sum_{\mathfrak{m}, \mathfrak{n} \in \mathfrak{M}} a (\mathfrak{m})
    b (\mathfrak{n}) \partial (\mathfrak{m}) \cdot \mathfrak{n}+
    \sum_{\mathfrak{m}, \mathfrak{n} \in \mathfrak{M}} a (\mathfrak{m}) b
    (\mathfrak{n}) \mathfrak{m} \cdot \partial (\mathfrak{n})\\
    & = & \sum_{\mathfrak{m}, \mathfrak{n} \in \mathfrak{M}} \partial (a
    (\mathfrak{m})\mathfrak{m}) \cdot (b (\mathfrak{n})\mathfrak{n}) +
    \sum_{\mathfrak{m}, \mathfrak{n} \in \mathfrak{M}} (a
    (\mathfrak{m})\mathfrak{m}) \cdot (b (\mathfrak{n}) \partial
    (\mathfrak{n}))\\
    & = & \left( \sum_{\mathfrak{m} \in \mathfrak{M}} \partial (a
    (\mathfrak{m})\mathfrak{m}) \right) \cdot b + a \cdot \left(
    \sum_{\mathfrak{n} \in \mathfrak{M}} \partial (b
    (\mathfrak{n})\mathfrak{n}) \right) \text{{\hspace*{\fill}}(by
    Proposition~\ref{prop-sum-prod})}\\
    & = & \partial (a) \cdot b + a \cdot \partial (b) .
    \text{{\hspace*{\fill}}(by strong linearity)}
  \end{eqnarray*}
  This concludes the proof.
\end{proof}

The following result is folklore. We prove it for completion.

\begin{proposition}
  \label{prop-Taylor-morphism}Assume that $\mathbf{K}$ has characteristic $0$.
  Let $\mathbf{A}$ be a $\mathbf{K}$-algebra and let $\partial \of \mathbf{A}
  \longrightarrow \mathbf{A}$ be a derivation. Then the function
  \begin{eqnarray}
    \mathcal{T}_{\partial} \of \mathbf{A} & \longrightarrow & \mathbf{A}
    \llbracket X \rrbracket \nonumber\\
    a & \longmapsto & \sum_{k \in \mathbb{N}} \frac{\partial^{[k]} (a)}{k!}
    X^k  \label{eq-Tpartial}
  \end{eqnarray}
  is a morphism of algebras.
\end{proposition}

\begin{proof}
  Let $a, b \in \mathbf{A}$. For $n \in \mathbb{N}$, an easy induction using
  the Leibniz product rule shows that
  \[ \partial^{[k]} (a \cdot b) = \sum_{i = 0}^k \binom{k}{i} \partial^{[i]}
     (a) \cdot \partial^{[k - i]} (b) . \]
  We have
  \begin{eqnarray*}
    \mathcal{T}_{\partial} (a \cdot b) & = & \sum_{k \in \mathbb{N}}
    \frac{\partial^{[k]} (a \cdot b)}{k!} X^k\\
    & = & \sum_{k \in \mathbb{N}} \left( \sum_{i = 0}^k \frac{1}{k! (k - i)
    !} \partial^{[i]} (a) \cdot \partial^{[k - i]} (b) \right) X^k\\
    & = & \sum_{k \in \mathbb{N}} \left( \sum_{m + p = k} \frac{1}{m!p!}
    \partial^{[m]} (a) \cdot \partial^{[p]} (b) \right) X^k\\
    & = & \left( \sum_{m \in \mathbb{N}} \frac{1}{m!} \partial^{[m]} (a) X^m
    \right) \cdot \left( \sum_{p \in \mathbb{N}} \frac{1}{p!} \partial^{[p]}
    (b) X^p \right)\\
    & = & \mathcal{T}_{\partial} (a) \cdot \mathcal{T}_{\partial} (b) .
  \end{eqnarray*}
  The function $\mathcal{T}_{\partial}$ is clearly $\mathbf{K}$-linear, so we
  are done.
\end{proof}

\subsection{Differential pre-logarithmic Hahn fields}

Let $\mathfrak{M}$ be a non-trivial, linearly ordered Abelian group and let
$\mathbf{K}$ be an ordered field. We write $\mathbb{S}= \mathbf{K} \llbracket
\mathfrak{M} \rrbracket$. Recall that $\mathbb{S}$ is an ordered field
extension of $\mathbf{K}$.

\begin{definition}
  Let $\ell \of (\mathfrak{M}, \cdot, 1, <) \longrightarrow (\mathbb{S}, +, 0,
  <)$ be an embedding of ordered groups. Then we say that $(\mathbb{S}, \ell)$
  is a {\tmem{{\tmstrong{pre-logarithmic Hahn field}}}}{\index{transserial
  field}}.
\end{definition}

This is a weaker version of the notion of pre-logarithmic section
{\cite[Definition~2.7]{KM11}} on $\mathbb{S}$.

\begin{remark}
  \label{rem-transserial-extension}Given an embedding of ordered groups
  $\log_{\mathbf{K}} \of (\mathbf{K}^{>}, \cdot, 1, <) \longrightarrow
  (\mathbf{K}, +, 0, <)$, the function $\ell$ extends {\cite[Lemmas~4.12 and
  5.1 and~Theorem~4.1]{Kuhl:ordexp}} into an embedding of ordered groups $\log
  \of (\mathbb{S}^{>}, \cdot, 1, <) \longrightarrow (\mathbb{S}, +, 0, <)$
  with
  \[ \log (c\mathfrak{m}(1 + \varepsilon)) = \ell (\mathfrak{m}) +
     \log_{\mathbf{K}} (c) + \sum_{k > 0} \frac{(- 1)^{k - 1}}{k}
     \varepsilon^k \]
  for all $\mathfrak{m} \in \mathfrak{M}$, $c \in \mathbf{K}^{>}$ and
  $\varepsilon \prec 1$. As a consequence of Proposition~\ref{prop-piecewise-analycity},
  the function $\log$ is analytic on $\mathbb{S}^{>}$ with $\tmop{Conv}
  (\log)_s = s +\mathbb{S}^{\prec}$ and $\log^{(k)} (s) = (- 1)^k  (k - 1)
  !s^{- k}$ for all $s \in \mathbb{S}^{>}$ and $k > 0$.
\end{remark}

\begin{definition}
  Assume that $(\mathbb{S}, \ell)$ is a pre-logarithmic Hahn field. Let
  $\partial \of \mathbb{S} \longrightarrow \mathbb{S}$ be a strongly linear
  derivation. We say that $(\mathbb{S}, \ell, \partial)$ is a
  {\tmem{{\tmstrong{differential pre-logarithmic Hahn field}}}} if we have
  $\partial (\ell (\mathfrak{m})) = \frac{\partial
  (\mathfrak{m})}{\mathfrak{m}}$ for all $\mathfrak{m} \in \mathfrak{M}$.
\end{definition}

\begin{remark}
  Given $s, t, u \in \mathbb{S}$, we will write $s = t + o (u)$ if $s - t
  \prec u$. Suppose that $s$ and $t$ have maximal common truncation $u$, and
  write $s = u + c_1 \mathfrak{m}_1 + o (\mathfrak{m}_1)$, $t = u + c_2
  \mathfrak{m}_2 + o (\mathfrak{m}_2)$ where $c_i \in \mathbf{K}$ and
  $\mathfrak{m}_i \in \mathfrak{M} \cup \{ 0 \}$. Then $s > t$ if and only if
  $c_1 \mathfrak{m}_1 + o (\mathfrak{m}_1) > c_2 \mathfrak{m}_2 + o
  (\mathfrak{m}_2)$ if and only if $c_1 \mathfrak{m}_1 > c_2 \mathfrak{m}_2$.
\end{remark}

Our next technical result Proposition~\ref{prop-Noeth-ext-der} is a version of
Proposition~\ref{prop-Noeth-ext-comp} for derivations, where a final segment
$\mathfrak{S}$ of $\mathfrak{M}$, we extend a strongly linear derivation on
$\mathbb{S}$ into a strongly linear derivation on $\mathbb{S} \llbracket X
\rrbracket_{\mathfrak{S}}$. To that end, we need to ``prepare'' $\mathfrak{S}$
with respect to some weak summability condition on its boundary. This is why
we need the following general lemma.

\begin{lemma}
  \label{lem:coinitial-sequence}Let $C$ be a convex subset of some truncation
  closed $L \subseteq \mathbb{S}$. Then there is a series $\varphi$ and a
  coinitial sequence in $C$ of the form $(\psi_i + c_i \mathfrak{o}_i)_{i <
  \kappa}$, where each $\psi_i$ is a truncation of $\varphi$, $c_i \in
  \mathbf{K}$, $\mathfrak{o}_i \in \mathfrak{M} \cup \{ 0 \}$. Moreover, we
  may assume that either $(\psi_i)_{i < \kappa}$ is injective, in which case
  $\mathfrak{o}_i \in \tmop{supp} (\varphi) \cup \{ 0 \}$ (in particular,
  $(\psi_i + r_i \mathfrak{o}_i)_{i < \kappa}$ is weakly summable), or $\psi_i
  = \varphi$ for all $i < \kappa$.
\end{lemma}

\begin{proof}
  The conclusion is trivial if $C$ has a minimum, so assume $C$ does not. Let
  $\varphi_0 \assign 0$. By induction, consider the set $T_{i, I}$ of the
  dominant terms of $\gamma - \varphi_i$ for $\gamma$ in some initial segment
  $I \subseteq C$: if there is $I$ so that $T_{i, I} = \{ t_i \}$ is a
  singleton, we let $\varphi_{i + 1} \assign \varphi_i + t_i$, otherwise we
  stop. At the limit stage, let $\varphi_i \assign \sum_{j < i} t_j$. The
  procedure stops at some ordinal $i = \lambda$ and we set $\varphi \assign
  \varphi_{\lambda}$.
  
  For any $\gamma_i \in C$, let $\psi_i$ be the maximal common truncation of
  $\varphi$ and $\gamma_i$ and write $\gamma_i = \psi_i + c_i \mathfrak{o}_i +
  o (\mathfrak{o}_i)$ where $c_i \in \mathbf{K}$ and $\mathfrak{o}_i \in
  \mathfrak{M} \cup \{ 0 \}$. By maximality of $\psi_i$, there is $\gamma_{i +
  1} \in C$ such that $\gamma_{i + 1} < \gamma_i$ and the maximal common
  truncation of $\gamma_{i + 1}$ and $\gamma_i$ is exactly $\psi_i$. By
  definition of the ordering, we must have $\psi_{i + 1} + c_{i + 1}
  \mathfrak{o}_{i + 1} + o (\mathfrak{o}_{i + 1}) < \psi_i + c_i
  \mathfrak{o}_i + o (\mathfrak{o}_i)$.
  
  It follows that there is a coinitial sequence $(\gamma_i)_{i < \kappa}$ such
  that for all $i < j$ we have $\psi_i + c_i \mathfrak{o}_i + o
  (\mathfrak{o}_i) > \psi_j + c_j \mathfrak{o}_j + o (\mathfrak{o}_j)$.
  Therefore, the sequence $(\psi_i + c_i \mathfrak{o}_i)$ is also coinitial
  with $C$. It ranges in $L$ because $L$ is truncation closed, thus it ranges
  in $C$ too.
  
  After extracting a subsequence, we may further assume that $(\psi_i)_{i <
  \kappa}$ is either constant or injective. In the former case, we must have
  $\psi_i = \varphi$, and we are done. In the latter, write $\psi_{i + 1} =
  \psi_i + d_i \mathfrak{p}_i + o (\mathfrak{p}_i)$. We must have $\gamma_i =
  \psi_i + c_i \mathfrak{o}_i > \psi_i + d_i \mathfrak{p}_i + o
  (\mathfrak{p}_i) \ni \gamma_{i + 1}$; if $\mathfrak{o}_i \succ
  \mathfrak{p}_i$, then $\psi_i + d_i \mathfrak{o}_i > \psi_i + | 2 d_i |
  \mathfrak{p}_i > \gamma_1$; if $\mathfrak{o}_i \prec \mathfrak{p}_i$, then
  $s_i < 0$ and so $\psi_i > \gamma_1$. Hence after possibly replacing $c_i
  \mathfrak{o}_i$ with $| 2 d_i | \mathfrak{p}_i$ or with $0$, we may assume
  that $\mathfrak{o}_i \in \tmop{supp} (\psi_{i + 1}) \cup \{ 0 \} \subseteq
  \tmop{supp} (\varphi) \cup \{ 0 \}$, and we are done. \ 
\end{proof}

\subsection{Cut extensions of derivations}\label{subsection-cut-extensions2}

Let $(\mathbb{S}, \partial, \ell)$ be a differential pre-logarithmic Hahn
field with $\mathbb{S}= \mathbf{K} \llbracket \mathfrak{M} \rrbracket$, such
that $\ell (\mathfrak{M})$ is truncation closed in $\mathbb{S}$. Let
$\mathfrak{S} \subseteq \mathfrak{M}$ be a final segment and consider the
corresponding algebra $\mathbb{S} \llbracket X \rrbracket_{\mathfrak{S}} =
\mathbf{K} \llbracket \mathfrak{M} \times X^{\mathbb{N}} \rrbracket$ for the
ordering $\prec_{\mathfrak{S}}$ of \Cref{subsection-subalgebras}. Note that
$\partial$ is a derivation $\mathbb{S} \longrightarrow \mathbb{S} \llbracket X
\rrbracket_{\mathfrak{S}}$. Consider the function
\begin{eqnarray*}
  \overline{\partial} \of \mathfrak{M} \times X^{\mathbb{N}} & \longrightarrow
  & \mathbb{S} \llbracket X \rrbracket_{\mathfrak{S}}\\
  \mathfrak{m}X^k & \longmapsto & \mathfrak{m}' X^k . \end{eqnarray*}
  
\begin{proposition}
  \label{prop-Noeth-ext-der}The function $\overline{\partial}$ is Noetherian.
\end{proposition}

\begin{proof}
  Let $(\mathfrak{m}_i X^{k_i})_{i \in \mathbb{N}}$ be a strictly
  $\prec_{\mathfrak{S}}$-decreasing sequence in $\mathfrak{M} \times
  X^{\mathbb{N}}$. This means, by definition, that $(k_i)_{i \in \mathbb{N}}$
  is weakly increasing and that there are monomials $\mathfrak{u}_i \succ
  \mathfrak{S}$ such that $\mathfrak{m}_{i + 1} \mathfrak{u}_i^{k_{i + 1} -
  k_i} \prec \mathfrak{m}_i$. Letting $\mathfrak{p}_{i + 1} \assign
  \mathfrak{u}_0^{k_1 - k_0} \cdots \mathfrak{u}_i^{k_{i + 1} - k_i}$,
  $\mathfrak{p}_0 \assign 1$, we find that $\mathfrak{m}_{i + 1}
  \mathfrak{p}_{i + 1} \prec \mathfrak{m}_i \mathfrak{p}_i$.
  
  Now pick some arbitrary $\mathfrak{n}_i \in \tmop{supp} \mathfrak{m}_i'$ for
  $i \in \mathbb{N}$. We claim that after taking a subsequence, the monomials
  $\mathfrak{n}_i \mathfrak{p}_i$ appear in the supports of some summable
  family. This implies that there are $i < j$ such that $\overset{}{}
  \mathfrak{n}_i \mathfrak{p}_i \succ \mathfrak{n}_j \mathfrak{p}_j$, and so
  $\mathfrak{n}_i \succ (\min_{i \leqslant n < j} \mathfrak{u}_n) ^{k_j - k_i}
  \mathfrak{n}_j$, thus $\mathfrak{n}_i X^{k_i} \succ_{\mathfrak{S}}
  \mathfrak{n}_j X^{k_j}$, proving that $(\mathfrak{m}_i' X^{k_i})_{i \in
  \mathbb{N}}$ is summable, and so that $\overline{\partial}$ is Noetherian.
  
  As a warm-up, observe that if $(k_i)_{i \in \mathbb{N}}$ is constant, then
  $\mathfrak{n}_i \mathfrak{p}_i =\mathfrak{n}_i$ is in the support of
  $\mathfrak{m}_i'$, and $(\mathfrak{m}_i')_{i \in \mathbb{N}}$ is summable by
  strong linearity of $\partial$.
  
  In the general case, observe that by strong linearity of $\partial$, the
  family
  \[ (\mathfrak{m}_i \mathfrak{p}_i)' =\mathfrak{m}_i' \mathfrak{p}_i
     +\mathfrak{m}_i \mathfrak{p}_i' =\mathfrak{p}_i  (\mathfrak{m}_i'
     +\mathfrak{m}_i \mathfrak{p}_i^{\dag}) \]
  is summable. Note moreover that $\mathfrak{n}_i \mathfrak{p}_i \in
  \tmop{supp} \mathfrak{m}_i' \mathfrak{p}_i$. We also have
  \[ \mathfrak{p}_{i + 1}^{\dagger} = (k_1 - k_0) \mathfrak{u}_0^{\dagger} +
     \cdots + (k_{i + 1} - k_i) \mathfrak{u}_i^{\dag} . \]
  If the sequence $(\ell (\mathfrak{u}_i))_{i \in \mathbb{N}}$ is weakly
  summable, then $(\mathfrak{u}_i^{\dagger})_{i \in \mathbb{N}}$ is weakly
  summable, so $ (\mathfrak{m}_i \mathfrak{p}_i')_{i \in \mathbb{N}} =
  ((\mathfrak{m}_i \mathfrak{p}_i) \mathfrak{p}_i^{\dagger})_{i \in
  \mathbb{N}}$ is summable, hence $(\mathfrak{m}_i' \mathfrak{p}_i)_{i \in
  \mathbb{N}}$ is summable, and we are done.
  
  If the sequence $(\ell (\mathfrak{u}_i))_{i \in \mathbb{N}}$ is not
  coinitial in $\ell (\mathfrak{M} \setminus \mathfrak{S})$, then there is
  $\mathfrak{u} \succ \mathfrak{S}$ such that $\mathfrak{u} \preccurlyeq
  \mathfrak{u}_i$ for all $i$. So we may assume that $\mathfrak{u}_i
  =\mathfrak{u}$ for all $i$, in which case $(\ell (\mathfrak{u}_i))_{i \in
  \mathbb{N}}$ is weakly summable, and we are done. Otherwise, we may replace
  $(\mathfrak{u}_i)_{i \in \mathbb{N}}$ with a subsequence of any other
  coinitial sequence in $\ell (\mathfrak{M} \setminus \mathfrak{S})$. Since
  $\ell (\mathfrak{M})$ is truncation closed and $\ell (\mathfrak{M} \setminus
  \mathfrak{S})$ is a convex subset of $\ell (\mathfrak{M})$, by
  Lemma~\ref{lem:coinitial-sequence}, we may choose the sequence so that $(\ell
  (\mathfrak{u}_i))_{i \in \mathbb{N}}$ is either weakly summable, or of the
  form $(\varphi + c_i \mathfrak{o}_i)_{i \in \mathbb{N}}$ and not weakly
  summable. Thus after taking a subsequence with $(\mathfrak{o}_i)_{i \in
  \mathbb{N}}$ strictly increasing and with $r_i < 0$.
  
  In the former case, we are done. In the latter, note that for any choice of
  non-zero $n_i \in \mathbb{N}$, there is another coinitial sequence
  $(\mathfrak{v}_i)_{i \in \mathbb{N}}$ in $\ell (\mathfrak{M} \setminus
  \mathfrak{S})$ such that $\ell (\mathfrak{v}_i) = \varphi + n_i c_i
  \mathfrak{o}_i$. If there are $j \in \mathbb{N}$ and infinitely many $i \in
  \mathbb{N}$ such that $\mathfrak{n}_i \in \tmop{supp} (\mathfrak{m}_i
  \varphi') \cup \tmop{supp} (\mathfrak{m}_i \mathfrak{o}_j')$, we note that
  $(\mathfrak{m}_i \mathfrak{p}_i  (\varphi' +\mathfrak{o}_j'))_{i \in
  \mathbb{N}}$ is summable, and we are done. Otherwise, we can choose $n_i$ so
  that $\mathfrak{n}_i \in \tmop{supp} (\mathfrak{m}_i \mathfrak{o}_j')$
  implies that $\mathfrak{n}_i \in \tmop{supp} (\mathfrak{m}_i'
  +\mathfrak{m}_i \mathfrak{p}_i^{\dag})$ and in particular $\mathfrak{n}_i
  \mathfrak{p}_i \in \tmop{supp} (\mathfrak{m}_i \mathfrak{p}_i)'$. With this
  choice, for each $i$ we have $\mathfrak{n}_i \mathfrak{p}_i \in (\tmop{supp}
  \mathfrak{m}_i \varphi') \cup (\tmop{supp} (\mathfrak{m}_i
  \mathfrak{p}_i)')$, and we are done.
\end{proof}

We thus have a strongly linear extension $\overline{\partial} \of \mathbb{S}
\llbracket X \rrbracket_{\mathfrak{S}} \longrightarrow \mathbb{S} \llbracket X
\rrbracket_{\mathfrak{S}}$ of $\overline{\partial}$. It follows from
Lemma~\ref{lem-strongly-lin-der} that $\overline{\partial}$ is a derivation. Define
\[ \mathfrak{S}^{-\rotatebox[origin=c]{180}{\small{\dag}}} \assign
   \{\mathfrak{m} \in \mathfrak{M} \suchthat
   \mathfrak{d}_{\mathfrak{m}^{\dag}} \in \mathfrak{S}^{- 1} \} . \]
\begin{lemma}
  The class $\mathfrak{S}^{-\rotatebox[origin=c]{180}{\tmem{\small{\dag}}}}$
  is a subgroup of $\mathfrak{M}$.
\end{lemma}

\begin{proof}
  For $\mathfrak{m}, \mathfrak{n} \in \mathfrak{S}^{-\rotatebox[origin=c]{180}{\small{\dag}}}$, we have
  $\mathfrak{d}_{(\mathfrak{m}\mathfrak{n}^{- 1})^{\dag}}
  =\mathfrak{d}_{\mathfrak{m}^{\dag} -\mathfrak{n}^{\dag}} \preccurlyeq \max
  (\mathfrak{d}_{\mathfrak{m}^{\dag}}, \mathfrak{d}_{\mathfrak{n}^{\dag}})$.
  We deduce since $\mathfrak{S}^{- 1}$ is an initial segment of $\mathfrak{M}$
  that $\mathfrak{d}_{(\mathfrak{m}\mathfrak{n}^{- 1})^{\dag}} \in
  \mathfrak{S}^{- 1}$, whence $\mathfrak{m}\mathfrak{n}^{- 1} \in
  \mathfrak{S}^{-\rotatebox[origin=c]{180}{\small{\dag}}}$.
\end{proof}

We write $\mathbb{S}_{[\mathfrak{S}]} \assign \mathbf{K} \left\llbracket
\mathfrak{S}^{-\rotatebox[origin=c]{180}{\small{\dag}}}
\right\rrbracket$, so $\mathbb{S}_{[\mathfrak{S}]}$ is a subfield of
$\mathbb{S}$.

\begin{proposition}
  \label{prop-contracting-der}Suppose that $\partial (\mathfrak{S}^-
  \;^{\rotatebox[origin=c]{180}{\ensuremath{\dag}}}) \subseteq
  \mathbb{S}_{[\mathfrak{S}]}$. Then the function
  \begin{eqnarray*}
    X \cdot \overline{\partial} \of \mathbb{S}_{[\mathfrak{S}]} \llbracket X
    \rrbracket_{\mathfrak{S}} & \longrightarrow & \mathbb{S}_{[\mathfrak{S}]}
    \llbracket X \rrbracket_{\mathfrak{S}}\\
    P & \longmapsto & \overline{\partial} (P) X
  \end{eqnarray*}
  is a strongly linear and contracting derivation.
\end{proposition}

\begin{proof}
  This is the restriction of a strongly linear function, so it is strongly
  linear. Since $\overline{\partial}$ is a derivation, so is $X \cdot
  \overline{\partial}$. Let $\mathfrak{m}X^k \in \mathfrak{S}^{-\rotatebox[origin=c]{180}{\small{\dag}}} \times X^{\mathbb{N}}$,
  and let $\mathfrak{n} \in \tmop{supp} \left( (X \cdot \overline{\partial})
  (\mathfrak{m}X^k) \right)$. So $\mathfrak{m} \neq 1$, and
  $\mathfrak{n}=\mathfrak{q}X^{k + 1}$ for a $\mathfrak{q} \in \tmop{supp}
  \mathfrak{m}'$. We want to show that $\mathfrak{n} \prec_{\mathfrak{S}}
  \mathfrak{m}X^k$. We have $\mathfrak{q} \preccurlyeq \mathfrak{m}'$, so
  $\mathfrak{q}\mathfrak{m}^{- 1} \preccurlyeq \mathfrak{m}^{\dag}$. We deduce
  since $\mathfrak{m} \in \mathfrak{S}^{-\rotatebox[origin=c]{180}{\small{\dag}}}$ that
  $\mathfrak{q}\mathfrak{m}^{- 1} \in \mathfrak{S}^{- 1}$, so
  $\mathfrak{n}=\mathfrak{q}X^{k + 1} \prec_{\mathfrak{S}} \mathfrak{m}X^k$.
\end{proof}

\begin{corollary}
  \label{cor-Taylor-setting}Assume that $\partial (\mathfrak{S}^{-\rotatebox[origin=c]{180}{\tmem{\small{\dag}}}}) \subseteq
  \mathbb{S}_{[\mathfrak{S}]}$. Then the function
  \begin{eqnarray*}
    \mathbb{S}_{[\mathfrak{S}]} & \longrightarrow &
    \mathbb{S}_{[\mathfrak{S}]} \llbracket X \rrbracket_{\mathfrak{S}}\\
    s & \longmapsto & \sum_{k \in \mathbb{N}} \frac{s^{(k)}}{k!} X^k
  \end{eqnarray*}
  is a well-defined and strongly linear morphism of algebras.
\end{corollary}

\begin{proof}
  We apply Proposition~\ref{prop-iterating-operator} to $X \cdot \overline{\partial}$.
  Since $\mathbb{S}_{[\mathfrak{S}]} \subseteq \mathbb{S}_{[\mathfrak{S}]}
  \llbracket X \rrbracket_{\mathfrak{S}}$, this shows that the restriction of
  $\mathcal{T}_{X \cdot \overline{\partial}} = \sum_{k \in \mathbb{N}}
  \frac{(X \cdot \overline{\partial})^{[k]}}{k!}$ to
  $\mathbb{S}_{[\mathfrak{S}]}$ is well-defined and strongly linear. We see
  with Proposition~\ref{prop-Taylor-morphism} that it preserves products.
\end{proof}

\section{Taylor expansions}\label{section-Taylor-expansions}

Our goal in this section is to study the convergence of Taylor expansions. We
fix an ordered field $\mathbf{K}$. Let $\mathfrak{M}, \mathfrak{N}$ be
non-trivial, linearly ordered Abelian groups. Let $(\mathbb{S}, \ell,
\partial)$ be a differential pre-logarithmic Hahn field with $\mathbb{S}=
\mathbf{K} \llbracket \mathfrak{M} \rrbracket$, write $\mathbb{T} \assign
\mathbf{K} \llbracket \mathfrak{N} \rrbracket$ and let $\triangle \of
\mathbb{S} \longrightarrow \mathbb{T}$ be a strongly linear morphism of
ordered rings. We also fix an $x \in \mathbb{S}^{\times}$, such that for all $\mathfrak{m} \in \mathfrak{M}$ the
following holds:
\begin{equation}
 (\mathfrak{m}^{\dag}
  \preccurlyeq x^{- 1} \wedge (\tmop{supp} \mathfrak{m}')^{\dag} \preccurlyeq
  x^{- 1}) \text{ or $(\mathfrak{m}^{\dag} \succ x^{- 1} \wedge (\tmop{supp}
  \mathfrak{m}')^{\dag} \asymp \mathfrak{m}^{\dag})$)\label{eq-spec-cond}} .
\end{equation}

\begin{remark}
  The condition (\ref{eq-spec-cond}) is satisfied for differential fields of
  transseries, including surreal numbers (Proposition~\ref{prop-near-support-No}), that
  are built upon a variable $x$ in a constructive way (see
  Lemma~\ref{lem-der-dichotomy}). We expect it is valid in most reasonable
  differential fields of transseries.
\end{remark}

Given $s \in \mathbb{S}$, and $\delta \in \mathbb{T}$, we study the
convergence of the Taylor series $\sum_{k \in \mathbb{N}} \frac{\triangle
(s^{(k)})}{k!} X^k \in \mathbb{T} \llbracket X \rrbracket$ at~$X = \delta$.
That is, we want to find conditions under which the family $(\triangle
(s^{(k)}) \delta^k)_{k \in \mathbb{N}}$ is summable. Our summability result is
as follows:

\begin{theorem}
  \label{th-Taylor-convergence}Let $\mathfrak{S} \subseteq \mathfrak{M}$ be a
  well-based subset. For all $\delta \in \mathbb{T}$ with $\delta \prec
  \triangle (x)$ and $\triangle (\mathfrak{m}^{\dag}) \delta \prec 1$ whenever
  $\mathfrak{m} \in \mathfrak{S}$, the family $(\triangle (\mathfrak{m}^{(k)})
  \delta^k)_{\mathfrak{m} \in \mathfrak{S} \wedge k \in \mathbb{N}}$ is
  summable.
\end{theorem}

\begin{remark}
  \label{rem-radius}Subject to the condition $\delta \prec \triangle (x)$, the
  domain of summability of the family $(\triangle (\mathfrak{m}^{(k)})
  \delta^k)_{\mathfrak{m} \in \tmop{supp} s \wedge k \in \mathbb{N}}$ is
  optimal. Indeed, let $f \in \mathbb{S}$. If each element $\mathfrak{m}$ in
  the support of $f$ is flat in the sense that $\mathfrak{m}^{\dag}
  \preccurlyeq x^{- 1}$, then the condition $(\triangle (\mathfrak{m}^{\dag}))
  \delta \prec 1$ is already implied by $\delta \prec \triangle (x)$. Suppose
  now that there is an $\mathfrak{m} \in \tmop{supp} f$ which is not flat.
  Then (\ref{eq-spec-cond}) implies that we have $(\mathfrak{m}^{(k)})^{\dag}
  \asymp \mathfrak{m}^{\dag}$ for each $k \in \mathbb{N}$. Therefore, for all
  $\delta$ with $\triangle (\mathfrak{m}^{\dag}) \delta \succcurlyeq 1$, we
  have $\triangle (\mathfrak{m}^{(k)}) \preccurlyeq \triangle
  (\mathfrak{m}^{(k + 1)}) \delta$, whence
  \[ \triangle (\mathfrak{m}) \preccurlyeq \triangle (\mathfrak{m}') \delta
     \preccurlyeq \triangle (\mathfrak{m}'') \delta^2 \preccurlyeq \cdots . \]
  This implies that the family $(\triangle (\mathfrak{m}^{(k)}) \delta^k)_{k
  \in \mathbb{N}}$ is not summable.
  
  In transseries and surreal numbers (taking $x$ as the unique monomial with
  derivative $1$), a slightly stronger version of (\ref{eq-spec-cond}) applies:
  in the case $\mathfrak{m}^{\dag} \preccurlyeq x^{- 1}$, we get $(\tmop{supp}
  \mathfrak{m}')^{\dag} \asymp x^{- 1}$. This entails that
  $(\mathfrak{m}^{(k)})^{\dag} \asymp x^{- 1}$, thus as in the previous
  argument $\triangle (\mathfrak{m}^{(k)}) \preccurlyeq \triangle
  (\mathfrak{m}^{(k + 1)}) \delta$ whenever $\delta \succcurlyeq \triangle
  (x)$. So in surreal numbers and transseries, the bound $\delta \prec \Delta
  (x)$ is also sharp. In order to prove the stronger version of
  (\ref{eq-spec-cond}), it suffices to use the fact that if
  $\mathfrak{m}^{\dag} \preccurlyeq x^{- 1}$, then there are an $r \in
  \mathbb{R}$ and a $\mathfrak{n} \in \mathfrak{M}$ such that
  $\mathfrak{n}^{\dag} \prec 1$ and $\mathfrak{m}= x^r \mathfrak{n}$.
\end{remark}

Given a fixed $\delta \in \mathbb{T}$ with $\delta \prec \triangle (x)$, let
$\mathfrak{S}_{\delta} \assign \{ \mathfrak{n} \in \mathfrak{N}: \mathfrak{n}
\succ \delta \}$, and write
\begin{eqnarray*}
  \mathfrak{M}_{\triangle, \delta} & \assign & \left( {\triangle^{\ast}}^{}
  (\mathfrak{S}_{\delta}) \right)^{-\rotatebox[origin=c]{180}{\small{\dag}}}= \{\mathfrak{m} \in
  \mathfrak{M} \suchthat \triangle (\mathfrak{m}^{\dag}) \delta \prec 1\},\\
  \mathbb{S}_{\triangle, \delta} & \assign & \mathbb{R} \llbracket
  \mathfrak{M}_{\triangle, \delta} \rrbracket .
\end{eqnarray*}
We call the partial map
\begin{eqnarray*}
  \mathcal{T}_{\delta} (\triangle) \of \mathbb{S} & \longrightarrow &
  \mathbb{T}\\
  s & \longmapsto & \sum_{k \in \mathbb{N}} \frac{\triangle (s^{(k)})}{k!}
  \delta^k .
\end{eqnarray*}
a Taylor deformation of $\triangle$. \Cref{th-Taylor-convergence} follows from
the following result:

\begin{proposition}
  \label{prop-Taylor-deformation}The class $\mathfrak{M}_{\triangle, \delta}$
  is a subgroup of $\mathfrak{M}$ and $\mathbb{S}_{\triangle, \delta}$ is a
  differential subfield of $\mathbb{S}$. The Taylor deformation
  $\mathcal{T}_{\delta} (\triangle) \of \mathbb{S}_{\triangle, \delta}
  \longrightarrow \mathbb{T}$ is a well-defined strongly linear morphism of
  ordered rings.
\end{proposition}

We show that Taylor deformations of $\triangle$ satisfy the same commutative
diagrams as $\triangle$ with respect to analytic arrows:

\begin{theorem}
  \label{th-Taylor-analytic}Let $\mathcal{A}_{\mathbb{S}},
  \mathcal{A}_{\mathbb{T}}$ be classes of analytic functions on
  $\mathbb{S}_{\triangle, \delta}^{>}$ and $\mathbb{T}^{>}$ respectively with
  $\mathcal{A}_{\mathbb{S}}' \subseteq \mathcal{A}_{\mathbb{S}}$ and
  $\mathcal{A}_{\mathbb{T}}' \subseteq \mathcal{A}_{\mathbb{T}}$. Assume that
  $\tmop{Conv} (f)_s \supseteq s +\mathbb{S}^{\prec s}$ and $\tmop{Conv} (g)_t
  \supseteq t +\mathbb{T}^{\prec t}$ for all $(f, g) \in
  \mathcal{A}_{\mathbb{S}} \times \mathcal{A}_{\mathbb{T}}$ and $(s, t) \in
  \mathbb{S}_{\triangle, \delta}^{>} \times \mathbb{T}^{>}$. Let $\Psi \of
  \mathcal{A}_{\mathbb{S}} \longrightarrow \mathcal{A}_{\mathbb{T}}$ be a map
  with $\Psi (f') = \Psi (f)'$ for all $f \in \mathcal{A}_{\mathbb{S}}$ and
  \[ \triangle (f (s)) = \Psi (f) (\triangle (s))
     \text{{\hspace{3em}}and{\hspace{3em}}$\partial (f (s)) = \partial (s) f'
     (s)$} \]
  for all $s \in \mathbb{S}_{\triangle, \delta}^{>}$. Then we have
  \[ T_{\delta} (\triangle) (f (s)) = \Psi (f) (T_{\delta} (\triangle) (s)) \]
  for all $s \in \mathbb{S}_{\triangle, \delta}^{>}$.
\end{theorem}

\begin{remark}
  \label{rem-log-compatibility}This applies in particular to
  $\mathcal{A}_{\mathbb{S}} = \{\log^{(k)} \suchthat k \in \mathbb{N}\}$ and
  $\mathcal{A}_{\mathbb{T}} = \{\log^{(k)} \suchthat k \in \mathbb{N}\}$ if
  $(\mathbb{T}, \ell)$ is itself a pre-logarithmic Hahn field and $\ell$ is
  extended to a logarithm $\log$ (see Remark~\ref{rem-transserial-extension}).
  Assume that $\log$ stabilises $\mathbb{S}_{\triangle, \delta}$. Then,
  setting $\Psi (\log^{(k)}) \assign \log^{(k)}$ for each $k \in \mathbb{N}$,
  we have
  \[ \partial (\log^{(k)} (s)) = \partial (s) \log^{(k + 1)} (s) \]
  for all $s \in \mathbb{S}^{>}$. Indeed for $k = 1$, this follows the
  definition of transserial derivations; for $k > 1$, this follows from the
  Leibniz rule, since $\log' (s) = s^{- 1}$. Then \Cref{th-Taylor-analytic}
  states that if $\triangle$ commutes with $\log$, then so do its Taylor
  deformations.
\end{remark}

\subsection{Convergence of Taylor expansions}

We first analyse convergence of series in $\mathbb{S}$. Recall that we fixed a
$\delta \in \mathbb{T}$ with $\delta \prec \triangle (x)$, and set
$\mathfrak{S}_{\delta} = \{\mathfrak{n} \in \mathfrak{N} \suchthat
\mathfrak{n} \succ \delta\}$, $\mathfrak{M}_{\triangle, \delta} = \left(
{\triangle^{\ast}}^{} (\mathfrak{S}_{\delta}) \right)^-
\;^{\rotatebox[origin=c]{180}{\ensuremath{\dag}}} = \{\mathfrak{m} \in
\mathfrak{M} \suchthat \triangle (\mathfrak{m}^{\dag}) \delta \prec 1\}$,
$\mathbb{S}_{\triangle, \delta} =\mathbb{R} \llbracket
\mathfrak{M}_{\triangle, \delta} \rrbracket$.

\begin{lemma}
  \label{lem-diff-subfield}We have $\partial (\mathfrak{M}_{\triangle,
  \delta}) \subseteq \mathbb{S}_{\triangle, \delta}$.
\end{lemma}

\begin{proof}
  Let $\mathfrak{m} \in \mathfrak{M}_{\triangle, \delta}$ and let
  $\mathfrak{n} \in \tmop{supp} \mathfrak{m}'$. If $\mathfrak{m}^{\dag}
  \preccurlyeq x^{- 1}$, then we have $\mathfrak{n}^{\dag} \preccurlyeq x^{-
  1}$ by (\ref{eq-spec-cond}). We deduce since $\delta \prec \triangle (x)$
  that $\triangle (\mathfrak{n}^{\dag}) \delta \prec 1$. If
  $\mathfrak{m}^{\dag} \succ x^{- 1}$, then (\ref{eq-spec-cond}) yields
  $\mathfrak{n}^{\dag} \asymp \mathfrak{m}^{\dag}$. We deduce since $\triangle
  (\mathfrak{m}^{\dag}) \delta \prec 1$ that $\triangle (\mathfrak{n}^{\dag})
  \delta \prec 1$. Thus $\tmop{supp} \mathfrak{m}' \subseteq
  \mathfrak{M}_{\triangle, \delta}$.
\end{proof}

Corollary~\ref{cor-Taylor-setting} applied to $\mathfrak{S}= \triangle^{\ast}
(\mathfrak{S}_{\delta})$ says that the Taylor morphism $\mathbb{S}_{\triangle,
\delta} \rightarrow \mathbb{S}_{\triangle, \delta} \llbracket X
\rrbracket_{\mathfrak{S}}$ is strongly linear, Proposition~\ref{prop-Noeth-ext-comp}
guarantees that we may hit the coefficients of the Taylor series with
$\triangle$ and obtain a series in $\mathbb{T} \llbracket X
\rrbracket_{\mathfrak{S}_{\delta}}$, and Proposition~\ref{prop-Noeth-evaluation} allows
us to substitute $\delta$ for $X$. We thus obtain
Proposition~\ref{prop-Taylor-deformation}. By Proposition~\ref{prop-contracting-der}, we
also obtain the following.

\begin{corollary}
  \label{cor-Taylor-descent}For $s \in \mathbb{S}_{\triangle, \delta}$, then
  we have
  \[ \triangle (s) \succ \triangle (s') \delta \succ \triangle (s'') \delta^2
     \succ \cdots . \]
\end{corollary}

\subsection{Taylor expansions and analytic functions}

We next prove \Cref{th-Taylor-analytic}. In order to do that, we rely on the
following formal result:

\begin{proposition}
  \label{prop-Taylor-expansion-hyperlog}Let $\mathfrak{M}_0, \mathfrak{N}_0$
  be non-trivial, linearly ordered Abelian groups and set $\mathbb{S}_0
  =\mathbb{R} \llbracket \mathfrak{M}_0 \rrbracket$ and $\mathbb{T}_0
  =\mathbb{R} \llbracket \mathfrak{N}_0 \rrbracket$. Let $\triangle_0 \of
  \mathbb{S}_0 \longrightarrow \mathbb{T}_0$ be a strongly linear morphism of
  rings and let $\mathd \of \mathbb{S}_0 \longrightarrow \mathbb{S}_0 \: ; \:
  s \mapsto s'$ be a strongly linear derivation.
  
  Let $s \in \mathbb{S}_0$, let $f \of \mathbb{S}_0 \longrightarrow
  \mathbb{S}_0$ and $g \of \mathbb{T}_0 \longrightarrow \mathbb{T}_0$ be
  analytic at $s$ and $\triangle_0 (s)$ respectively, with $\tmop{Conv} (f)_s
  \supseteq s +\mathbb{S}_0^{\prec s}$ and $\tmop{Conv} (g)_{\triangle_0 (s)}
  \supseteq \triangle_0 (s) +\mathbb{T}_0^{\prec \triangle_0 (s)}$. Assume
  that for all $k \in \mathbb{N}$, we have
  \begin{equation}
    (f^{(k)} (s))' = f^{(k + 1)} (s) s' \label{eq-chain-rule-abstract}
  \end{equation}
  and
  \begin{equation}
    \triangle_0 (f^{(k)} (s)) = g^{(k)} (\triangle_0 (s))
    \label{eq-cond-prop-hyperlog-Taylor-comp} .
  \end{equation}
  Let $\varepsilon \in \mathbb{T}_0$ such that the family $(\triangle_0
  (s^{(k)}) \varepsilon^k)_{k \in \mathbb{N}}$ is summable with
  \begin{equation}
    \forall k > 0, \triangle_0 (s) \succ \triangle_0 (s^{(k)}) \varepsilon^k .
    \label{eq-cond2-prop-hyperlog-Taylor-comp}
  \end{equation}
  Then the family $(\triangle_0 (f (s)^{(k)}) \varepsilon^k)_{k \in
  \mathbb{N}}$ is summable, with
  \[ \sum_{k \in \mathbb{N}} \frac{\triangle_0 (f (s)^{(k)})}{k!}
     \varepsilon^k = g \left( \sum_{k \in \mathbb{N}} \frac{\triangle_0
     (s^{(k)})}{k!} \varepsilon^k \right) . \]
  In other words, the relation $\triangle_0 \circ f = g \circ \triangle_0$ is
  also satisfied for the Taylor deformation $T_{\varepsilon} (\triangle_0) \of
  s \mapsto \sum_{k \in \mathbb{N}} \frac{\triangle_0 (s^{(k)})}{k!}
  \varepsilon^k$ of $\triangle_0$.
\end{proposition}

\begin{proof}
  We may assume that $\varepsilon \neq 0$. By
  Proposition~\ref{prop-analicity-strong}, the function $\mathcal{A} \of
  \mathbb{S}^{\preccurlyeq \varepsilon}_0 \longrightarrow \mathbb{T}_0$ given
  for $\delta \preccurlyeq \varepsilon$ by
  \begin{eqnarray*}
    \mathcal{A} (\delta) & \assign & \sum_{k \in \mathbb{N}} \frac{\triangle_0
    (s^{(k)})}{k!} \delta^k
  \end{eqnarray*}
  is analytic on $\mathbb{S}^{\preccurlyeq \varepsilon}_0$. Our goal is to
  show that $g (\mathcal{A} (\varepsilon)) = \tilde{P} (\varepsilon)$ where
  \[ P \assign \sum_{k \in \mathbb{N}} \frac{\triangle_0 (f (s)^{(k)})}{k!}
     X^k \in \mathbb{T} \llbracket X \rrbracket . \]
  The function $g$ is analytic at $\triangle_0 (s)$ with $\tmop{Conv}
  (g)_{\triangle_0 (s)} \supseteq \triangle_0 (s) +\mathbb{T}_0^{\prec
  \triangle_0 (s)}$. For $n \in \mathbb{N}$ and $k > 0$, we set
  \begin{eqnarray*}
    X_{n, k} & \assign & \{ v \in (\mathbb{N}^{>})^n \suchthat | v | \assign
    v_{[1]} + \cdots + v_{[n]} = k \} \text{\quad and}\\
    c_{k, n} & \assign & \sum_{v \in X_{n, k}} \frac{g^{(n)} (\triangle_0
    (s))}{n!}  \frac{\triangle_0 (s^{(v_{[1]})})}{v_{[1]} !} \cdots
    \frac{\triangle_0 (s^{(v_{[n]})})}{v_{[n]} !} .
  \end{eqnarray*}
  We have $\mathcal{A} (\delta) - \triangle_0 (s) \prec \triangle_0 (s)$ by
  (\ref{eq-cond2-prop-hyperlog-Taylor-comp}), so we may apply
  Proposition~\ref{cor-Taylor-composition} and see that $g \circ \mathcal{A}$
  is analytic on $\mathbb{S}_0^{\preccurlyeq \varepsilon}$. Moreover, the
  family $(c_{k, n} \varepsilon^k)_{n \in \mathbb{N}, k > 0}$ is summable,
  with
  \begin{equation}
    g \circ \mathcal{A} (\varepsilon) = g (\triangle_0 (s)) + \sum_{n \in
    \mathbb{N}, k > 0} c_{n, k} \varepsilon^k . \label{eq-Taylor-hyperlog-cnk}
  \end{equation}
  So by Lemma~\ref{lem-Fubini}, the family $\left( \sum_{n \in \mathbb{N}}
  c_{k, n} \varepsilon^k \right)_{k > 0}$ is summable, and
  \[ \sum_{n \in \mathbb{N}, k > 0} c_{n, k} \varepsilon^k = \sum_{k > 0}
     \left( \sum_{n \in \mathbb{N}} c_{k, n} \right) \varepsilon^k . \]
  Since $g (\triangle_0 (s)) = \triangle_0 (f (s))$ and in view of
  (\ref{eq-Taylor-hyperlog-cnk}), it suffices to show that $\sum_{n \in
  \mathbb{N}} c_{k, n} = \frac{f (s)^{(k)}}{k!}$ for all $k > 0$. By
  (\ref{eq-cond-prop-hyperlog-Taylor-comp}), we have $\triangle_0 (f^{(n)}
  (s)) = g^{(n)} (\triangle_0 (s))$ for all $n \in \mathbb{N}$. Recall that we
  have a chain rule (\ref{eq-chain-rule-abstract}) at $s$. An induction gives
  Fa{\`a} di Bruno's formula, i.e.
  \[ \frac{(f (s))^{(k)}}{k!} = \sum_{n \in \mathbb{N}} \sum_{v \in X_{n, k}}
     \frac{f^{(n)} (s)}{n!}  \frac{s^{(v_{[1]})}}{v_{[1]} !} \cdots
     \frac{s^{(v_{[n]})}}{v_{[n]} !} . \]
  Therefore
  \[ \frac{\triangle_0 (f (s)^{(k)})}{k!} = \sum_{n \in \mathbb{N}} \sum_{v
     \in X_{n, k}} \frac{g^{(n)} (\triangle_0 (s))}{n!}  \frac{\triangle_0
     (s^{(v_{[1]})})}{v_{[1]} !} \cdots \frac{\triangle_0
     (s^{(v_{[n]})})}{v_{[n]} !} = \sum_{n \in \mathbb{N}} c_{n, k} . \]
  This concludes the proof.
\end{proof}

\Cref{th-Taylor-analytic} follows from Proposition~\ref{prop-Taylor-expansion-hyperlog}
for $(\mathbb{S}_0, \mathbb{T}_0, \triangle_0, \mathd) = (\mathbb{S},
\mathbb{T}, \triangle, \partial)$ and $g \assign \Psi (f)$ for each $f \in
\mathcal{A}_{\mathbb{S}}$. Just as `commutative diagrams' are preserved by
Taylor deformations, so are `chain rules' in the following sense:

\begin{theorem}
  \label{th-chain-rule}Assume that $x' = 1$. Let $\mathd \of \mathbb{T}
  \longrightarrow \mathbb{T}$ be a strongly linear derivation such that
  \[ \forall s \in \mathbb{S}, \mathd (\triangle (s)) = \mathd (\triangle (x))
     \triangle (s') . \]
  Then for all $s \in \mathbb{S}_{\triangle, \delta}$, we have $\mathd
  (T_{\delta} (\triangle) (s)) = \mathd (T_{\delta} (\triangle) (x))
  T_{\delta} (\triangle) (s')$.
\end{theorem}

\begin{proof}
  The relation $\mathd \circ \triangle = \mathd (\triangle (x)) \triangle
  \circ \partial$ gives
  \[ \mathd \circ \triangle \circ \partial^{[k]} = \mathd (\triangle (x))
     \triangle \circ \partial^{[k + 1]} \]
  for all $k \in \mathbb{N}$. Let $s \in \mathbb{S}_{\triangle, \delta}$. We
  have
  \begin{eqnarray*}
    \mathd (T_{\delta} (\triangle) (s)) & = & \mathd \left( \sum_{k \in
    \mathbb{N}} \frac{\triangle (s^{(k)})}{k!} \delta^k \right)\\
    & = & \mathd (\delta)  \sum_{k > 0} \frac{\triangle (s^{(k)})}{k!} k
    \delta^{k - 1} + \sum_{k \in \mathbb{N}} \frac{\mathd (\triangle
    (s^{(k)}))}{k!} \delta^k\\
    & = & \mathd (\delta)  \sum_{k > 0} \frac{\triangle (s^{(k)})}{k!} k
    \delta^{k - 1} + \sum_{k \in \mathbb{N}} \frac{\mathd (\triangle (x))
    \triangle (s^{(k)})}{k!} \delta^k\\
    & = & \mathd (\delta)  \sum_{k > 0} \frac{\triangle (s^{(k)})}{(k - 1) !}
    \delta^{k - 1} + \mathd (\triangle (x))  \sum_{k \in \mathbb{N}}
    \frac{\triangle (s^{(k + 1)})}{k!} \delta^k\\
    & = & \mathd (\delta + \triangle (x))  \sum_{k \in \mathbb{N}}
    \frac{\triangle (s^{(k + 1)})}{k!} \delta^k\\
    & = & \mathd (T_{\delta} (\triangle) (x)) T_{\delta} (\triangle) (s') .
    \text{{\hspace*{\fill}}(as $x' = 1$)}
  \end{eqnarray*}
  This concludes the proof.
\end{proof}

\subsection{Application to
\texorpdfstring{$\omega$}{ω}-series}\label{subsection-application-w-series}

The field $\mathbf{No} =\mathbb{R} \llbracket \mathbf{Mo} \rrbracket$, with
Gonshor's logarithm $\log$ {\cite[Chapter~10]{Gon86}}, is a transseries field
as per {\cite[Definition 2.2.1]{Schm01}}. It is also a differential
pre-logarithmic field for Berarducci and Mantova's derivation $\partial$
of~{\cite{BM18}}. For $\mathfrak{m}, \mathfrak{n} \in \mathbf{Mo}$, we have
$\log \mathfrak{m} \prec \log \mathfrak{n}$ if and only if
$\mathfrak{m}^{\dag} \prec \mathfrak{n}^{\dag}$, so in view of
{\cite[Proposition~2.2.4(1)]{Schm01}}, we have:

\begin{lemma}
  \label{lem-flatness-log}For all $\mathfrak{m} \in \mathbf{Mo}$, we have
  $(\tmop{supp} \ell (\mathfrak{m}))^{\dag} \prec \mathfrak{m}^{\dag}$.
\end{lemma}

Kuhlmann and Matusinski isolated {\cite[Section~4]{KM15}} surreal monomials
$\kappa_{- \gamma, n}, \gamma \in \mathbf{On}, n \in \mathbb{N}$ which later
played a particular role in the definition of $\partial$. Consider the class
$\mathfrak{W} \subseteq \mathbf{Mo}$ of infinitesimal monomials
\[ \mathfrak{l}_{\alpha} \assign \exp \left( - \sum_{\gamma < \alpha} \sum_{n
   \in \mathbb{N}} \kappa_{- \gamma, n + 1} \right), \]
where $\alpha$ ranges in $\mathbf{On}$. We have $\alpha < \beta
\Longrightarrow \mathfrak{l}_{\alpha} \succ \mathfrak{l}_{\beta}$ for all
$\alpha, \beta \in \mathbf{On}$, so $\mathfrak{W}$ is well-based. Moreover, we
$\mathfrak{l}_{\alpha}^{\dag} = - \sum_{\gamma < \alpha} \sum_{n \in
\mathbb{N}} \partial (\kappa_{- \gamma, n + 1}) \sim \partial (\kappa_{0, 1})
= \partial (\log \omega) = \omega^{- 1}$ for all $\alpha \in \mathbf{On}$,
whence~$\mathfrak{W}^{\dag} \preccurlyeq \omega^{- 1}$. See
{\cite[Section~5.3]{BM18}} and {\cite[Section~2]{vdH:bm}} for more details.

\begin{proposition}
  \label{prop-near-support-No}For each $\mathfrak{m} \in \mathbf{Mo} \setminus
  \{1\}$ and $\mathfrak{n} \in \tmop{supp} \partial (\mathfrak{m})$, there are
  an $\mathfrak{s} \in \mathbf{Mo}$ with $\mathfrak{s}^{\dag} \prec
  \mathfrak{m}^{\dag}$ and $\mathfrak{s} \succcurlyeq 1$ and a $\mathfrak{w}
  \in \mathfrak{W}$ with $\mathfrak{n}=\mathfrak{m}\mathfrak{s}\mathfrak{w}$.
\end{proposition}

\begin{proof}
  The definition of $\partial \of \mathbf{No} \longrightarrow \mathbf{No}$,
  relies {\cite[Definition~6.11]{BM18}} on the notion of path in transseries
  fields {\cite{Schm01}}. A (finite) path in a monomial $\mathfrak{m} \in
  \mathbf{Mo}$ is a sequence $(r_i \mathfrak{m}_i)_{i \leqslant k}$ where $r_0
  = 1$, $\mathfrak{m}_0 =\mathfrak{m}$ and each $r_{i + 1} \mathfrak{m}_{i +
  1}$ for $i < k$ is a positive infinite term in $\tmop{supp} \log
  \mathfrak{m}_i$. Note that $\mathfrak{m}_0^{\dag} \succ
  \mathfrak{m}_1^{\dag} \succ \cdots \succ \mathfrak{m}_k^{\dag}$ by
  Lemma~\ref{lem-flatness-log}.
  
  Fix $\mathfrak{m} \in \mathbf{Mo}$ and $\mathfrak{n} \in \tmop{supp}
  \partial (\mathfrak{m})$. By {\cite[Definitions~6.13 and~6.7]{BM18}}, there
  is a path $(r_i \mathfrak{m}_i)_{i \leqslant k}$ in $\mathfrak{m}$ and an
  $\mathfrak{l}_{\alpha} \in \mathfrak{W}$ with
  $\mathfrak{n}=\mathfrak{m}\mathfrak{m}_1 \cdots \mathfrak{m}_k
  \mathfrak{l}_{\alpha}$. Since $\mathfrak{m}_1^{\dag}, \ldots,
  \mathfrak{m}_k^{\dag} \prec \mathfrak{m}_0^{\dag} =\mathfrak{m}^{\dag}$, we
  have $(\mathfrak{m}_1 \cdots \mathfrak{m}_k)^{\dag} \prec \mathfrak{m}$.
  Since $\mathfrak{m}_1, \ldots, \mathfrak{m}_n \in \tmop{supp} \log
  \mathfrak{m}$, we have $\mathfrak{m}_1 \cdots \mathfrak{m}_n \succcurlyeq
  1$. This concludes the proof.
\end{proof}

\begin{lemma}
  \label{lem-der-dichotomy}Let $\mathbb{S}= \mathbf{K} \llbracket \mathfrak{M}
  \rrbracket$ be a field of well-based series equipped with a strongly linear
  derivation $\mathbb{S} \longrightarrow \mathbb{S} \: ; \: s \mapsto s'$ and
  let $x \in \mathbb{S}^{\times}$. Assume that there is a class $\mathfrak{W}
  \subseteq \mathfrak{M}$ such that $\mathfrak{W}^{\dag} \preccurlyeq x^{-
  1}$, and that for all $\mathfrak{m} \in \mathfrak{M} \setminus \{1\}$ and
  $\mathfrak{n} \in \tmop{supp} \mathfrak{m}'$ there are an $\mathfrak{s} \in
  \mathfrak{M}$ and a $\mathfrak{w} \in \mathfrak{W}$ such that
  $\mathfrak{s}^{\dag} \prec \mathfrak{m}^{\dag}$ and
  $\mathfrak{n}=\mathfrak{m}\mathfrak{s}\mathfrak{w}$. Then the condition
  {\tmem{(\ref{eq-spec-cond})}} is satisfied with respect to $x$.
\end{lemma}

\begin{proof}
  We may assume that $\mathfrak{m} \neq 1$. Let $\mathfrak{n} \in \tmop{supp}
  \mathfrak{m}'$, and let $\mathfrak{s} \in \mathfrak{M}$ and $\mathfrak{w}
  \in \mathfrak{W}$ such that $\mathfrak{s}^{\dag} \prec \mathfrak{m}^{\dag}$
  and $\mathfrak{n}=\mathfrak{m}\mathfrak{s}\mathfrak{w}$. So
  $\mathfrak{n}^{\dag} =\mathfrak{m}^{\dag} +\mathfrak{s}^{\dag}
  +\mathfrak{w}^{\dag}$. If $\mathfrak{m}^{\dag} \preccurlyeq x^{- 1}$, then
  $\mathfrak{s}^{\dag} \prec x^{- 1}$ so $\mathfrak{n}^{\dag} \preccurlyeq
  x^{- 1}$. If $\mathfrak{m}^{\dag} \succ x^{- 1}$, then $\mathfrak{w}^{\dag}
  \prec \mathfrak{m}^{\dag}$ so $\mathfrak{n}^{\dag} -\mathfrak{m}^{\dag}
  \prec \mathfrak{m}^{\dag}$, whence in particular~$\mathfrak{n}^{\dag} \asymp
  \mathfrak{m}^{\dag}$. \end{proof}

\begin{corollary}
  Let $\mathfrak{M} \subseteq \mathbf{No}$ be a subgroup and assume that
  $\mathbb{R} \llbracket \mathfrak{M} \rrbracket$ is a differential subfield
  of $(\mathbf{No}, \partial)$ containing $\omega$. Then $(\mathbb{R}
  \llbracket \mathfrak{M} \rrbracket, \partial, \omega)$ satisfies
  {\tmem{(\ref{eq-spec-cond})}}.
\end{corollary}

Let $\mathbb{R} \llangle \omega \rrangle$ be the field of $\omega$-series as
defined in {\cite[Definition~4.7]{BM19}}. This is the smallest subfield of
$\mathbf{No}$ containing $\omega$ which is closed under $\exp$, $\log$ and all
sums of summable families. In particular $(\mathbb{R} \llangle \omega
\rrangle, \partial, \omega)$ satisfies (\ref{eq-spec-cond}). Given $a \in
\mathbf{No}^{>, \succ}$, we write have a function $\mathord{\circ_a} \of \mathbb{R} \llangle \omega \rrangle \longrightarrow
   \mathbf{No} \: ; \: f \mapsto f \circ a$,
where $\circ$ is the composition law $\mathord{\circ} \of \mathbb{R} \llangle
\omega \rrangle \times \mathbf{No}^{>, \succ} \longrightarrow \mathbf{No}$
defined in {\cite{BM19}}.

\begin{proposition}
  Let $a, \delta \in \mathbf{No}$ with $a >\mathbb{R}$ and $\delta \prec a$.
  Then the function
  \begin{eqnarray*}
    \mathcal{T}_{\delta} (\circ_a) \of \mathbb{R} \llangle \omega
    \rrangle_{\circ_a, \delta} & \longrightarrow & \mathbf{No}\\
    s & \longmapsto & \sum_{k \in \mathbb{N}} \frac{s^{(k)} \circ a}{k!}
    \delta^k
  \end{eqnarray*}
  coincides with $\circ_{a + \delta}$ on $\mathbb{R} \llangle \omega
  \rrangle_{\circ_a, \delta}$.
\end{proposition}

\begin{proof}
  We claim that the logarithm stabilises $\mathbb{R} \llangle \omega
  \rrangle_{\circ_a, \delta}$. Indeed, it suffices to show that $\log
  \mathfrak{m} \in \mathbb{R} \llangle \omega \rrangle_{\circ_a, \delta}$
  whenever $\mathfrak{m} \in \mathbb{R} \llangle \omega \rrangle_{\circ_a,
  \delta}^{>}$ is a monomial. We have $\tmop{supp} \log \mathfrak{m} \succ 1$
  by construction, so it suffices to show that $((\log \mathfrak{m})^{\dag}
  \circ a) \delta \prec 1$. But $(\log \mathfrak{m})^{\dag} = \frac{(\log \log
  \mathfrak{m})'}{\log \mathfrak{m}} = \frac{\mathfrak{m}^{\dag}}{(\log
  \mathfrak{m})^2} \prec \mathfrak{m}^{\dag}$, so $((\log \mathfrak{m})^{\dag}
  \circ a) \delta \prec (\mathfrak{m}^{\dag} \circ a) \delta \prec 1$. This
  proves our claim.
  
  Given a $b \in \mathbf{No}^{>, \succ}$, the function $\mathord{\circ_b} \of
  \mathbb{R} \llangle \omega \rrangle \longrightarrow \mathbf{No}$ is the
  unique strongly linear morphism of rings with
  \[ \mathord{\circ_b} (\omega) = b \text{{\hspace{3em}}and{\hspace{3em}}}
     \forall f \in \mathbb{R} \llangle \omega \rrangle^{> 0},
     \mathord{\circ_b} (\log f) = \log (\mathord{\circ_b} (f)) . \]
  Since $\omega \in \mathbb{R} \llangle \omega \rrangle_{\circ_a, \delta}$ and
  in view of Lemma~\ref{lem-flatness-log}, the field $\mathbb{R} \llangle \omega
  \rrangle_{\circ_a, \delta}$ is a transserial subfield of $\mathbb{R}
  \llangle \omega \rrangle$. Therefore the function $\mathord{\circ_b}
  \upharpoonleft \mathbb{R} \llangle \omega \rrangle_{\circ_a, \delta}$ is
  also unique to satisfy the above conditions on $\mathbb{R} \llangle \omega
  \rrangle_{\circ_a, \delta}$. Recall that $\circ_a$ itself commutes with the
  logarithm. In view of Remark~\ref{rem-log-compatibility},
  \Cref{th-Taylor-analytic} implies that $\mathcal{T}_{\delta} (\circ_a)$
  commutes with log. We conclude by observing that
  \[ \mathcal{T}_{\delta} (\circ_a) (\omega) = \omega \circ a + (\omega' \circ
     a) \delta + \cdots = a + 1 \cdot \delta + 0 + 0 + \cdots = a + \delta .
  \]
\end{proof}

This establishes Theorem~\ref{th-main}.

\end{document}